\documentclass[11pt]{article}

\usepackage{url}
\usepackage{amssymb,enumerate}
\usepackage{amsmath,amsfonts}
\usepackage{amsthm}
\usepackage{latexsym}
\usepackage{mathrsfs}
\usepackage{color}
\usepackage{microtype}
\usepackage{geometry}
\usepackage{algorithm}
\usepackage{algorithmicx}
\usepackage{algpseudocode}

\usepackage{hyperref}

\usepackage{multirow}
\allowdisplaybreaks[0]
\usepackage[capitalise,nameinlink]{cleveref}
\crefformat{equation}{(#2#1#3)}
\usepackage{xcolor}
\hypersetup{backref=false,       
	pagebackref=true,               
	hyperindex=true,                
	colorlinks=true,                
	breaklinks=true,                
	urlcolor= black,                
	linkcolor= black,                
	bookmarks=true,                 
	bookmarksopen=false,
	filecolor=black,
	citecolor=blue,
	linkbordercolor=blue
}

\usepackage{amsmath,amsfonts,bm}
\usepackage{enumerate}
\def\eqref#1{equation~\ref{#1}}
\def\1{\bm{1}}

\DeclareMathAlphabet{\mathsfit}{\encodingdefault}{\sfdefault}{m}{sl}
\SetMathAlphabet{\mathsfit}{bold}{\encodingdefault}{\sfdefault}{bx}{n}

\def\gO{{\mathcal{O}}}

\newcommand{\E}{\mathbb{E}}

\newcommand{\R}{\mathbb{R}}

\allowdisplaybreaks[4]

\theoremstyle{plain}
\newtheorem{theorem}{Theorem}

\newtheorem{lemma}{Lemma}

\newtheorem{assumption}{Assumption}
\newtheorem{proposition}{Proposition}

\theoremstyle{definition}

\newtheorem{remark}{Remark}

\usepackage{graphicx, color}
\usepackage{algorithm, algpseudocode} % use algorithm and algorithmicx for typesetting algorithms
\usepackage{mathrsfs} % for \mathscr command

\usepackage{lipsum}

\title{A stochastic first-order method with multi-extrapolated momentum for highly smooth unconstrained optimization}

\author{
Chuan He\thanks{Department of Mathematics, Link\"oping University, Sweden (email: {\tt chuan.he@liu.se}). This work was partially supported by the Wallenberg AI, Autonomous Systems and Software Program (WASP) funded by the Knut and Alice Wallenberg Foundation.}
}
\date{
\today
}

\begin{document}
\maketitle
	
\begin{abstract}
In this paper, we consider unconstrained stochastic optimization problems in which the objective function exhibits higher-order smoothness. In particular, we propose a stochastic first-order method (SFOM) with multi-extrapolated momentum, where multiple extrapolations are performed in each iteration, followed by a momentum update based on these extrapolations. We show that the proposed SFOM can accelerate optimization by exploiting the higher-order smoothness of the objective function $f$. Assuming that the $p$th-order derivative of $f$ is Lipschitz continuous for some $p\ge2$, and under additional mild assumptions, we establish that our method achieves a sample and first-order operation complexity of $\widetilde{\gO}(\epsilon^{-(3p+1)/p})$ for finding an $\epsilon$-stochastic stationary point $x$ that satisfies $\E[\|\nabla f(x)\|]\le\epsilon$. To the best of our knowledge, this is the first SFOM to leverage the Lipschitz continuity of arbitrarily higher-order derivatives of the objective function for acceleration.
%, resulting in a sample complexity that improves upon the best-known results without assuming the mean-squared smoothness condition. 
Preliminary numerical experiments validate the practical performance of our method and support our theoretical results.
\end{abstract}


\medskip


\section{Introduction}\label{sec:intro}
In this paper, we consider the smooth unconstrained optimization problem:
\begin{align}\label{pb:uc}
\min_{x\in\R^n} f(x), 
\end{align}
where $f:\R^n\to\R$ is continuously differentiable and has a Lipschitz continuous $p$th-order derivative for some $p\ge2$ (see \cref{asp:pth-smth} for details). We assume that problem \cref{pb:uc} has at least one optimal solution. Our goal is to develop a first-order method for solving \cref{pb:uc} in the stochastic regime, where the derivatives of $f$ are not directly accessible. Instead, the algorithm relies solely on the stochastic estimator $G(\cdot;\xi)$ for $\nabla f(\cdot)$, which is unbiased and has bounded variance (see Assumption \ref{asp:basic}(c) for details). Here, $\xi$ is a random variable with a sample space $\Xi$.

In recent years, there have been significant developments in stochastic first-order methods (SFOMs) with complexity guarantees for solving problem~\cref{pb:uc}. Notably, when assuming $\nabla f$ is Lipschitz continuous (see \cref{asp:basic}(b)), SFOMs \cite{cutkosky2020momentum,ghadimi2013stochastic,ghadimi2016accelerated} have been proposed with a sample and first-order operation complexity\footnote{Sample complexity and first-order operation complexity refer to the total number of samples of $\xi$ and the evaluations of $G(\cdot,\cdot)$ throughout the algorithm, respectively.} of $\mathcal{O}(\epsilon^{-4})$ for finding an $\epsilon$-stochastic stationary point $x$ of problem \cref{pb:uc} such that 
\begin{align}\label{expect-eps}
\E[\|\nabla f(x)\|]\le\epsilon,
\end{align}
where $\epsilon\in(0,1)$ is a given tolerance parameter, and the expectation is taken over the randomness in the algorithm. This complexity bound has been proven to be optimal in \cite{arjevani2023lower} without any additional smoothness assumptions beyond the Lipschitz continuity of the gradient. Among these works, \cite{cutkosky2020momentum} proposed a normalized SFOM that incorporates Polyak momentum updates:
\begin{align}\label{polyak-mom}
m^k = (1-\gamma_{k-1}) m^{k-1} + \gamma_{k-1}G(x^k;\xi^k)\qquad\forall k\ge0,     
\end{align}
where $\{x^k\}$ denotes the algorithm iterates, $\{\gamma_k\}$ denotes the momentum parameters, and $\{m^k\}$ denotes the estimators of $\{\nabla f(x^k)\}$. Recently, \cite{gao2024non} proposed an SFOM with Polyak momentum for stochastic composite optimization. As shown in \cite{cutkosky2020momentum,gao2024non}, Polyak momentum induces a variance reduction effect in gradient estimation for SFOMs. %, and it was fu shown in \cite{cutkosky2020momentum} that Polyak momentum facilitates the convergence of SFOMs with normalized updates. %Other beneficial theoretical properties of SFOMs with Polyak momentum have been studied in \cite{jelassi2022towards,li2022last,sebbouh2021almost,wang2020quickly}.

Several other attempts have been made to improve the sample complexity of SFOMs by leveraging the second-order smoothness of $f$. Assuming that $\nabla^2 f$ is Lipschitz continuous, i.e., \cref{asp:pth-smth} with $p=2$, an SFOM \cite{cutkosky2020momentum} has been proposed with sample and first-order operation complexity of $\mathcal{O}(\epsilon^{-7/2})$ for finding $x$ satisfying \cref{expect-eps}. This method incorporates implicit gradient transport that combines extrapolations with Polyak momentum updates:
\begin{align}\label{extra-p-moment}
z^k = x^k + \frac{1-\gamma}{\gamma}(x^k-x^{k-1}),\quad m^k = (1-\gamma) m^{k-1} + \gamma G(z^k;\xi^k)\qquad\forall k\ge0
\end{align}
with some fixed parameter $\gamma\in(0,1)$. It has been shown in \cite{cutkosky2020momentum} that constructing $\{z^k\}$ through extrapolations and combining them with Polyak momentum achieves faster variance reduction for gradient estimators $\{m^k\}$, leading to an improved overall complexity bound. In addition, a restarted stochastic gradient method has been proposed in \cite{fang2019sharp} to escape saddle points for problem \cref{pb:uc} assuming a Lipschitz continuous Hessian, which achieves a sample and first-order operation complexity of $\mathcal{O}(\epsilon^{-7/2})$ for finding an $(\epsilon,\mathcal{O}(\sqrt{\epsilon}))$-second-order stationary point of problem \cref{pb:uc} with high probability. However, the algorithm input parameters in \cite{cutkosky2020momentum, fang2019sharp} (e.g., step sizes and momentum parameters) require explicit access to the problem-specific parameters (e.g., Lipschitz constants and bounds for noise), and is therefore generally hard to implement in practice. Furthermore, there appear to be no SFOMs that leverage higher-order smoothness of $f$ beyond the Hessian Lipschitz condition for acceleration.

Aside from leveraging the smoothness of $f$ for acceleration, recent work \cite{cutkosky2019momentum,fang2018spider,li2021page} has shown that SFOMs can be accelerated by leveraging the mean-squared smoothness of the stochastic gradient estimator $G(\cdot;\xi)$, which is expressed as
\begin{align}\label{asp:ave-smt}
\E_{\xi}[\|G(y;\xi) - G(x;\xi)\|^2]\le L^2 \|y-x\|^2\qquad\forall x,y\in\R^n  
\end{align}
for some $L>0$. These methods achieve a sample and first-order operation complexity of $\mathcal{O}(\epsilon^{-3})$ for finding an $\epsilon$-stochastic stationary point of \cref{pb:uc}, which has been proven to be optimal in \cite{arjevani2023lower} under the assumption of \cref{asp:ave-smt}. In addition, SFOMs \cite{pham2020proxsarah,tran2022hybrid,wang2019spiderboost,xu2023momentum} have been proposed for stochastic composite optimization problems, achieving a sample and first-order operation complexity of $\mathcal{O}(\epsilon^{-3})$ under mean-squared smoothness condition stated in \cref{asp:ave-smt}. Thus far, the mean-squared smoothness in \cref{asp:ave-smt} appears to be crucial for achieving the complexity of $\mathcal{O}(\epsilon^{-3})$ for finding $x$ that satisfies \cref{expect-eps}. 
On the other hand, it should be noted that the mean-squared smoothness is an assumption imposed on the stochastic gradient estimator $G(\cdot;\xi)$ rather than the function $f$ itself. To illustrate this, we provide an example in \cref{apx:ctexp} that shows an unbiased stochastic gradient estimator $G(\cdot; \xi)$ with bounded variance can violate \cref{asp:ave-smt} when the gradient and arbitrarily higher-order derivatives of $f$ are Lipschitz continuous.

Recently, many studies have also extended the aforementioned acceleration methods for unconstrained optimization to develop accelerated SFOMs for constrained optimization with complexity guarantees. For example, \cite{alacaoglu2024complexity,huang2019faster,li2024stochastic,lu2024variance,shi2025momentum} establish the complexity for equality constrained stochastic optimization problems, while \cite{he2024stochastic} establishes the complexity for conic constrained stochastic optimization. Since the main focus of this paper is on unconstrained optimization, we omit a detailed discussion of these studies.

In this paper, we explore the possibility to accelerate SFOMs for solving problem \cref{pb:uc} by exploiting the higher-order smoothness of the objective function $f$. In particular, we show that SFOMs can achieve acceleration by {\it exploiting arbitrarily higher-order smoothness of $f$} through the introduction of an SFOM with multi-extrapolated momentum (Algorithm \ref{alg:unf-sfom}). %Our complexity analysis is based on establishing the descent for the potential functions, which differs significantly from the analysis \cite{cutkosky2020momentum}. As a result, our 
%Our proposed SFOM can be viewed as a significant generalization of the SFOM introduced in \cite{cutkosky2020momentum}, which employs extrapolated momentum steps described in \cref{extra-p-moment}. %, as the acceleration of our method also relies on extrapolation and momentum. 
We demonstrate that for any $p\ge2$, performing $p-1$ separate extrapolations in each iteration and combining them with a Polyak momentum update can accelerate variance reduction by exploiting the smoothness of the $p$th-order derivative of $f$, thereby leading to a sample and first-order operation complexity of $\widetilde{\mathcal{O}}(\epsilon^{-(3p+1)/p})$\footnote{$\widetilde{\mathcal{O}}(\cdot)$ represents ${\mathcal{O}}(\cdot)$ with hidden logarithmic factors.}. Our complexity result significantly generalizes the complexity of $\mathcal{O}(\epsilon^{-4})$ assuming Lipschitz continuous gradients \cite{cutkosky2020momentum, ghadimi2013stochastic, ghadimi2016accelerated}, and the complexity of $\mathcal{O}(\epsilon^{-7/2})$ assuming Lipschitz continuous Hessians \cite{cutkosky2020momentum, fang2019sharp}. Moreover, our complexity analysis is based on establishing a descent relation for the potential functions, which significantly differs from \cite{cutkosky2020momentum}, allowing our algorithm to employ time-varying input parameters (e.g., step sizes and momentum parameters) without explicit knowledge of the problem-specific parameters (e.g., Lipschitz constants and bounds for noise). Consequently, given that the algorithm input parameters in \cite{cutkosky2020momentum} require explicit access to the problem-specific parameters, our algorithm is easier to implement for problems with Lipschitz continuous Hessians.
%Moreover, unlike the straightforward parameter choices in previous SFOMs, the input parameters of our proposed SFOM are determined through an innovative application of Lagrange interpolation (see Section \ref{sec:method}). %and the resulting parameters satisfy the elegant property of alternating signs . 
For ease of comparison, we summarize the sample and first-order operation complexity of several existing SFOMs, along with their associated smoothness assumptions, as well as those of our method, in Table \ref{table:sum-sc}.

\begin{table}[h!]
\centering
\caption{Comparison of sample and first-order operation complexity of several SFOMs and their associated smoothness assumptions in the literature with those of our method for finding an $\epsilon$-stochastic stationary point of \cref{pb:uc}. Here, SG and PM stand for ``stochastic gradient" and ``Polyak momentum", respectively.}
\smallskip
\begin{tabular}{l|l|l}
\hline
Method  & Complexity & Smoothness assumption \\
\hline
SG \cite{ghadimi2013stochastic} & ${\mathcal{O}}(\epsilon^{-4})$ & gradient Lipschitz \\
SG-PM \cite{cutkosky2020momentum} & ${\mathcal{O}}(\epsilon^{-4})$ &  gradient Lipschitz \\
%Restarted SG \cite{fang2019sharp}  & ${\mathcal{O}}(\epsilon^{-7/2})$ & gradient \& Hessian Lipschitz \\
NIGT \cite{cutkosky2020momentum} & ${\mathcal{O}}(\epsilon^{-7/2})$ & gradient \& Hessian Lipschitz \\
{\bf \cref{alg:unf-sfom} (ours)} &  $\widetilde{\mathcal{O}}(\epsilon
^{-(3p+1)/p})$ & gradient \& $p$th-order derivative Lipschitz \\
STORM \cite{cutkosky2019momentum}  &  ${\mathcal{O}}(\epsilon^{-3})$ & mean-squared smooth \\
SPIDER \cite{fang2018spider} & ${\mathcal{O}}(\epsilon^{-3})$ & mean-squared smooth \\
PAGE \cite{li2021page} & ${\mathcal{O}}(\epsilon^{-3})$ & mean-squared smooth \\
\hline
\end{tabular}
\label{table:sum-sc}
\end{table}

The main contributions of this paper are highlighted below.

\begin{itemize}
\item We propose an SFOM with multi-extrapolated momentum (\cref{alg:unf-sfom}). This method the first SFOM to leverage the Lipschitz continuity of {\it arbitrarily higher-order derivatives} of the objective function for acceleration, providing insights for future algorithmic design. %Our method is efficient to implement in practice, with update schemes and parameter selection that can be performed cheaply and neatly (see \cref{sec:method}), offering insights for future algorithmic design.

\item We show that, assuming that the $p$th-order derivative of $f$ is Lipschitz continuous and under other mild assumptions, our proposed method achieves a sample and first-order operation complexity of $\widetilde{\mathcal{O}}(\epsilon^{-(3p+1)/p})$ for finding an $\epsilon$-stochastic stationary point of problem \cref{pb:uc}. This complexity result improves upon the best-known results for SFOMs without assuming mean-squared smoothness.
\end{itemize}

The rest of this paper is organized as follows. In \cref{sec:not-asm}, we introduce some notation, assumptions, and preliminaries that will be used in the paper. In \cref{sec:method}, we propose an SFOM with multi-extrapolated momentum and establish its complexity. \cref{sec:ne} presents preliminary numerical results. 
In \cref{sec:proof,sec:cr}, we present the proofs of the main results, and some concluding remarks, respectively.

\section{Notation, assumptions, and preliminaries}\label{sec:not-asm}

Throughout this paper, let $\R^n$ denote the $n$-dimensional Euclidean space and $\langle\cdot,\cdot\rangle$ denote the standard inner product. We use $\|\cdot\|$ to denote the Euclidean norm of a vector or the spectral norm of a matrix. For any $p\ge1$ and a $p$th-order continuously differentiable function $\varphi$, we denote by $D^p\varphi(x)[h_1,\ldots,h_p]$ the $p$th-order directional derivative of $\varphi$ at $x$ along $h_i\in\R^n$, $1\le i\le p$, and use $D^p \varphi(x)[\cdot]$ to denote the associated symmetric $p$-linear form. For any symmetric $p$-linear form $\mathcal{T}[\cdot]$, we denote its norm as
\begin{align}\label{def:pnorm}
\|\mathcal{T}\|_{(p)}\doteq\max_{h_1,\ldots,h_p}\{\mathcal{T}[h_1,\ldots,h_p]:\|h_i\|\le1,1\le i\le p\}. 
\end{align}
For any $x\in\R^n$ and $h_i\in\R^n$ with $1\le i\le p-1$, we define $\nabla^p\varphi(x)(h_1,\ldots,h_{p-1})\in\R^n$ as follows:
\[
\langle \nabla^p\varphi(x)(h_1,\ldots,h_{p-1}), h_p\rangle \equiv D^p\varphi(x)[h_1,\ldots,h_p]\qquad \forall h_p\in\R^n.
\]
For any $x,h\in\R^n$, we denote $D^p \varphi(x)[h]^{p}\doteq D^p \varphi(x)[h,\ldots,h]$ and $\nabla^p \varphi(x)(h)^{p-1}\doteq\nabla^p \varphi(x)(h,\ldots,h)$. For any $s\in\R$, we let $\mathrm{sgn}(s)$ be $1$ if $s\ge0$ and $-1$ otherwise. In addition, $\widetilde{\mathcal{O}}(\cdot)$ represents $\mathcal{O}(\cdot)$ with logarithmic terms omitted.

We now make the following assumptions throughout this paper.
\begin{assumption}\label{asp:basic}
\begin{enumerate}[{\rm (a)}]
\item There exists a finite $f_{\mathrm{low}}$ such that $f(x)\ge f_{\mathrm{low}}$ for all $x\in\R^n$.    
\item There exists $L_1>0$ such that $\|\nabla f(y) - \nabla f(x)\| \le L_1\|y-x\|$ holds for all $x,y\in\R^n$.
\item The stochastic gradient estimator $G:\R^n\times \Xi\to\R^n$ satisfies 
\begin{align}\label{est-unbias-varbd}
\E_{\xi}[G(x;\xi)]=\nabla f(x),\quad \E_{\xi}[\|G(x;\xi) - \nabla f(x)\|^2]\le\sigma^2\qquad \forall x\in\R^n
\end{align}
for some $\sigma>0$.
\end{enumerate}
\end{assumption}

We now make some remarks on \cref{asp:basic}.

\begin{remark}
Assumptions \ref{asp:basic}(a) and (b) are standard. It follows from \cref{asp:basic}(b) that
\begin{align}\label{ineq:1st-desc}
f(y) \le f(x) + \nabla f(x)^T(y-x) + \frac{L_1}{2}\|y-x\|^2\qquad \forall x,y\in\R^n.    
\end{align}
In addition, \cref{asp:basic}(c) implies that $G(\cdot;\xi)$ is an unbiased stochastic estimator for $\nabla f(\cdot)$ with bounded variance, which is commonly used in stochastic optimization (e.g., see \cite{bottou2018optimization,curtis2017optimization,lan2020first}).
\end{remark}

We also make the following assumption regarding the Lipschitz continuity of $D^pf$ for some $p\ge2$.
\begin{assumption}\label{asp:pth-smth}
The function $f$ is $p$th-order continuously differentiable for some $p\ge2$, and moreover, there exists $L_p>0$ such that $\|D^p f(y) - D^p f(x)\|_{(p)} \le L_p\|y-x\|$ holds for all $x,y\in\R^n$.
\end{assumption}

The following lemma provides a useful inequality under \cref{asp:pth-smth}, and its proof is deferred to \cref{sec:proof-lem12}.

\begin{lemma}\label{lem:pth-desc}
Under \cref{asp:pth-smth}, the following inequality holds:
\begin{align}\label{ineq:pth-desc}
\Big\|\nabla f(y) -\sum_{r=1}^p\frac{1}{(r-1)!}\nabla^r f(x)(y-x)^{r-1}\Big\|\le \frac{L_p}{p!}\|y-x\|^p\qquad\forall x,y\in\R^n.
\end{align}
\end{lemma}

\section{A stochastic first-order method with multi-extrapolated momentum}\label{sec:method}
In this section, we propose an SFOM with multi-extrapolated momentum in \cref{alg:unf-sfom} and establish its sample and first-order operation complexity under Assumptions \ref{asp:basic} and \ref{asp:pth-smth}.

Specifically, at the $k$th iteration, our method performs $q$ separate extrapolations for some $q\ge1$ as in \cref{update-zk} to obtain $q$ points $\{z^{k,t}\}_{1\le t\le q}$, where the extrapolation parameters $\{\gamma_{k-1,t}\}_{1\le t\le q}$ in \cref{update-zk} are chosen to have distinct positive values. Then, a gradient estimator $m^k$ for $\nabla f(x^k)$ is constructed by combining the previous gradient estimator $m^{k-1}$ with the stochastic gradient estimator $G(\cdot;\xi^k)$ evaluated at $\{z^{k,t}\}_{1\le t\le q}$, as described in \cref{update-mk}. To exploit the higher-order smoothness of $f$, the values of the weighting parameters $\{\theta_{k-1,t}\}_{1\le t\le q}$ in \cref{update-mk} need to be determined by solving the linear system in \cref{pth-lss}, with the coefficient matrix constructed using $\{\gamma_{k-1,t}\}_{1\le t\le q}$. %The resulting values of $\{\theta_{k-1,t}\}_{1\le t\le q}$ follow a pattern of alternating signs (see \cref{lem:p-sol}), given that the values of $\{\gamma_{k-1,t}\}_{1\le t\le q}$ are ordered. 
After obtaining $m^k$, the next iterate $x^{k+1}$ is generated via a normalized update\footnote{Normalized updates are recognized as an important technique in training deep neural networks (see, e.g., \cite{You2020Large}).}, as described in \cref{update-xk}. For ease of understanding our extrapolations and momentum updates, we visualize the updates of $\{z^{k,t}\}_{1\le t\le 3}$ and $m^k$ generated by \cref{alg:unf-sfom} with $q=3$ on a two-dimensional contour plot, as shown in \cref{visual}.

\begin{figure}[ht]
\centering
\begin{minipage}[b]{0.49\linewidth}
\centering
\includegraphics[width=\linewidth]{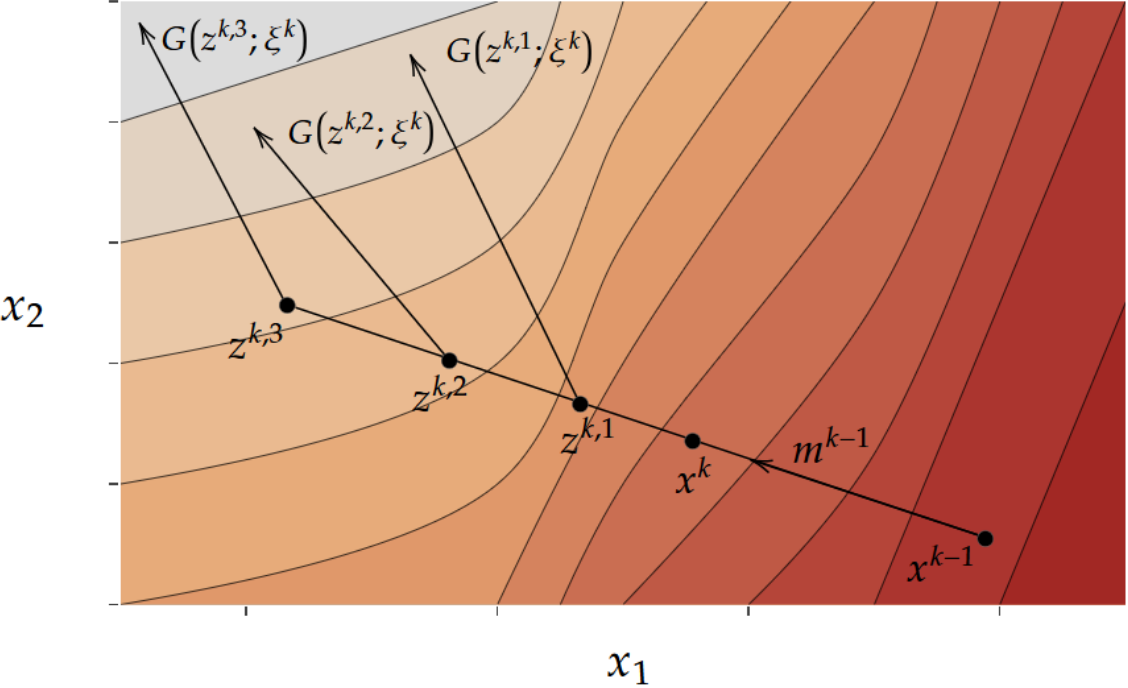}
\end{minipage}
\hfill
\begin{minipage}[b]{0.49\linewidth}
\centering
 \hfill\includegraphics[width=\linewidth]{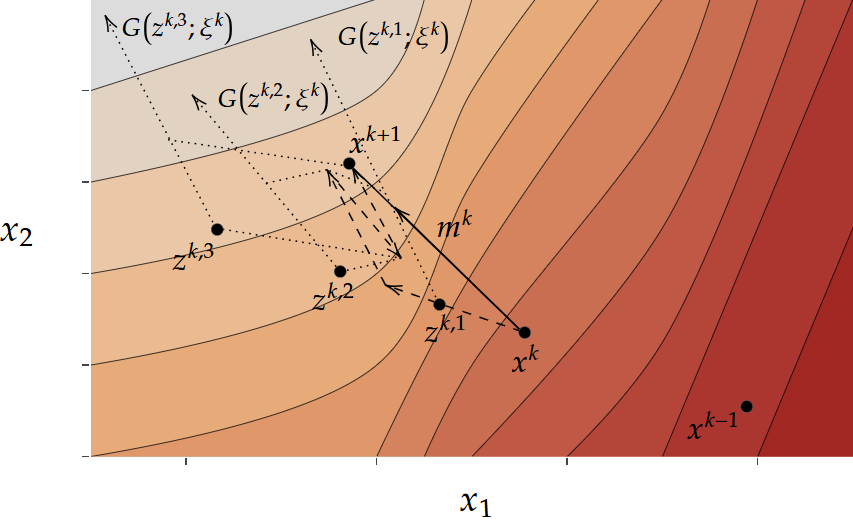}
\end{minipage}
\caption{Visualization of the updates for $\{z^{k,t}\}_{1\le t\le 3}$ (left) and $m^k$ (right) on a contour plot.}
\label{visual}
\end{figure}

Before proceeding, we present the following lemma regarding the descent relation of $f$ on iterates generated by \cref{alg:unf-sfom}. Its proof is deferred to \cref{sec:proof-lem12}.

\begin{lemma}\label{lem:1-desc}
Suppose that Assumptions \ref{asp:basic} and \ref{asp:pth-smth} hold. Let $\{(x^k,m^k)\}_{k\ge 0}$ be generated by \cref{alg:unf-sfom} with step sizes $\{\eta_k\}_{k\ge0}$. Then, 
\begin{align}\label{ineq:1st-desc-1}
f(x^{k+1}) \le f(x^k) - \eta_k\|\nabla f(x^k)\| + 2\eta_k\|\nabla f(x^k) - m^k\| + \frac{L_1}{2}\eta_k^2\qquad\forall k\ge0,
\end{align}
where $L_1$ is given in \cref{asp:basic}(b).
\end{lemma}

%interesting phenomenon +-+-.
\subsection{An SFOM with double-extrapolated momentum}\label{subsec:sfom-dem}
In this subsection, we analyze a simple variant of \cref{alg:unf-sfom} with $q=2$, which we refer to as the SFOM with double-extrapolated momentum, since two separate extrapolations are performed in each iteration. In the following, we establish its sample and first-order operation complexity under Assumption \ref{asp:basic} and \cref{asp:pth-smth} with $p=3$.

Throughout this subsection, we impose the following equations on the algorithm input parameters $\{(\gamma_{k,t},\theta_{k,t})\}_{1\le t\le 2,k\ge0}$:
\begin{align}\label{3rd-lss}
\theta_{k,1}/\gamma_{k,1} + \theta_{k,2}/\gamma_{k,2} = 1,\quad \theta_{k,1}/\gamma_{k,1}^2 + \theta_{k,2}/\gamma_{k,2}^2 = 1\qquad\forall k\ge0.
\end{align}  
It can be verified that for any $\gamma_{k,1}$ and $\gamma_{k,2}$ with two distinct positive values, the values of $\theta_{k,1}$ and $\theta_{k,2}$ can be uniquely determined by solving the above equations. In addition, we require that
\begin{align}\label{3rd-theta-01}
\theta_{k,1} + \theta_{k,2} \in(0,1)\qquad \forall k\ge 0.    
\end{align}
The specific selection of $\{(\gamma_{k,t},\theta_{k,t})\}_{1\le t\le 2,k\ge0}$ for \cref{alg:unf-sfom} with $q=2$ such that \cref{3rd-lss} and \cref{3rd-theta-01} hold is given in \cref{eta-gamma-3} and \cref{theta-3} below. We next provide some useful recurrence properties for \cref{alg:unf-sfom} with $q=2$ assuming that \cref{3rd-lss} and \cref{3rd-theta-01} hold. The following lemma establishes the recurrence for the estimation error of the gradient estimators $\{m^k\}_{k\ge0}$ generated by \cref{alg:unf-sfom} with $q=2$, whose proof is deferred to \cref{subsec:proof-dem}. 

\begin{algorithm}[!htbp]
\caption{An SFOM with multi-extrapolated momentum}
\label{alg:unf-sfom}
\begin{algorithmic}[0]
\State \textbf{Input:} starting point $x^0\in\R^n$, step sizes $\{\eta_k\}_{k\ge0}\subset(0,+\infty)$, extrapolations per iteration $q$, extrapolation parameters $\{\gamma_{k,t}\}_{1\le t\le q,k\ge0}\subset(0,1)$, weighting parameters $\{\theta_{k,t}\}_{1\le t\le q,k\ge0}$ with $\sum_{t=1}^q\theta_{k,t}\in(0,1)$ for all $k\ge0$. 
\State Initialize $x^{-1}=x^0$, $m^{-1}=0$, and $(\gamma_{-1,t},\theta_{-1,t})=(1,1/q)$ for all $1\le t\le q$.
\For{$k=0,1,2,\ldots$}
\State Perform $q$ separate extrapolations: 
\begin{align}\label{update-zk}
z^{k,t} = x^k + \frac{1-\gamma_{k-1,t}}{\gamma_{k-1,t}}(x^k - x^{k-1})\qquad \forall 1\le t\le q.
\end{align}
\State Compute the search direction:
\begin{align}\label{update-mk}
m^k = \Big(1 - \sum_{t=1}^{q}\theta_{k-1,t}\Big) m^{k-1} + \sum_{t=1}^{q}\theta_{k-1,t}G(z^{k,t};\xi^{k}).
\end{align}
\State Update the next iterate:
\begin{align}\label{update-xk}
x^{k+1} = x^k - \eta_k \frac{m^k}{\|m^k\|}. 
\end{align}
\EndFor
\end{algorithmic}
\end{algorithm}

%For now, we assume that \cref{3rd-lss} and \cref{3rd-theta-01} hold without specifying their selection in order to establish the general convergence guarantees for \cref{alg:unf-sfom} with $q=2$. The specific selection of the input parameters $\{(\eta_k,\gamma_{k,1},\gamma_{k,2},\theta_{k,1},\theta_{k,2})\}_{k\ge0}$ that satisfy \cref{3rd-lss} and \cref{3rd-theta-01} are provided later in \cref{eta-gamma-3} and \cref{theta-3}. 

\begin{lemma}\label{lem:3rd-rec}
Suppose that \cref{asp:basic} holds, and \cref{asp:pth-smth} holds with $p=3$. Let $\{(x^k,m^k)\}_{k\ge0}$ be generated by \cref{alg:unf-sfom} with inputs $q=2$, step sizes $\{\eta_k\}_{k\ge0}$, and $\{(\gamma_{k,t},\theta_{k,t})\}_{1\le t\le 2,k\ge0}$ satisfying \cref{3rd-lss} and \cref{3rd-theta-01}. Then, 
\begin{align}
&\mathbb{E}_{\xi^{k+1}}[\|m^{k+1} - \nabla f (x^{k+1})\|^2]\le (1-\theta_{k,1}-\theta_{k,2})\|m^k - \nabla f (x^k)\|^2\nonumber\\
&\qquad + \frac{L_3^2\eta_k^6\theta_{k,1}^2}{12\gamma_{k,1}^6(\theta_{k,1}+\theta_{k,2})} + \frac{L_3^2\eta_k^6\theta_{k,2}^2}{12\gamma_{k,2}^6(\theta_{k,1}+\theta_{k,2})} + \frac{L_3^2\eta_k^6}{12(\theta_{k,1}+\theta_{k,2})} + 2(\theta_{k,1}^2+\theta_{k,2}^2) \sigma^2\qquad\forall k\ge0,\label{rec:3rd-m}
\end{align}
where $\sigma$ and $L_3$ are given in Assumptions \ref{asp:basic}(b) and \ref{asp:pth-smth}, respectively.
\end{lemma}

The following theorem establishes the descent relation of potentials $\{P_k\}_{k\ge0}$ for \cref{alg:unf-sfom} with $q=2$, whose proof is relegated to \cref{subsec:proof-dem}.

\begin{theorem}\label{thm:3rd-conv}
Suppose that \cref{asp:basic} holds, and \cref{asp:pth-smth} holds with $p=3$. Let $\{(x^k,m^k)\}_{k\ge0}$ be generated by \cref{alg:unf-sfom} with inputs $q=2$, step sizes $\{\eta_k\}_{k\ge0}$, and $\{(\gamma_{k,t},\theta_{k,t})\}_{1\le t\le 2,k\ge0}$ satisfying \cref{3rd-lss} and \cref{3rd-theta-01}. Define
\begin{align}\label{def:3rd-Pk}
P_k \doteq f(x^k) + p_k\|m^k - \nabla f(x^k)\|^2\qquad\forall k\ge0,
\end{align}
where $\{p_k\}_{k\ge0}$ is a positive sequence satisfying
\begin{align}\label{thetak-3rd}
(1-\theta_{k,1}-\theta_{k,2})p_{k+1}\le (1-(\theta_{k,1}+\theta_{k,2})/2)p_{k}\qquad\forall k\ge0.    
\end{align}
Then,
\begin{align}
&\E_{\xi^{k+1}}[P_{k+1}]\le P_k - \eta_k\|\nabla f(x^k)\| + \frac{L_1}{2}\eta_k^2 + \frac{2\eta_k^2}{(\theta_{k,1} + \theta_{k,2})p_k}\nonumber\\
&\qquad + \frac{L_3^2\eta_k^6\theta_{k,1}^2p_{k+1}}{12\gamma_{k,1}^6(\theta_{k,1} + \theta_{k,2})} + \frac{L_3^2\eta_k^6\theta_{k,2}^2p_{k+1}}{12\gamma_{k,2}^6(\theta_{k,1} + \theta_{k,2})} + \frac{L_3^2\eta_k^6p_{k+1}}{12(\theta_{k,1} + \theta_{k,2})} + 2(\theta_{k,1}^2+\theta_{k,2}^2)p_{k+1}\sigma^2\qquad\forall k\ge 0,\label{3rd-stat-upbd}  
\end{align}
where $L_1$ and $\sigma$ are given in \cref{asp:basic}, and $L_3$ is given in \cref{asp:pth-smth}.
\end{theorem}

The following theorem establishes an upper bound for the number of iterations of \cref{alg:unf-sfom} with $q=2$ before finding an $\epsilon$-stochastic stationary point of problem \cref{pb:uc}. Its proof is deferred to \cref{subsec:proof-dem}.

\begin{theorem}\label{cor:rate-3}
Suppose that \cref{asp:basic} holds, and \cref{asp:pth-smth} holds with $p=3$. Define
\begin{align}\label{M3}
M_3\doteq 4\big( f(x^0) - f_{\mathrm{low}} + 19\sigma^2 + L_1  + 4L_3^2 + 2\big),
\end{align}
where $f_{\mathrm{low}}$, $L_1$, and $\sigma$ are given in \cref{asp:basic}, and $L_3$ is given in \cref{asp:pth-smth}. Let $\{x^k\}_{k\ge0}$ be all iterates generated by \cref{alg:unf-sfom} with inputs $q=2$ and $\{(\eta_k,\gamma_{k,t},\theta_{k,t})\}_{1\le t\le 2,k\ge0}$ specified as
\begin{align}
&\eta_k = \frac{1}{(k+3)^{7/10}},\quad \gamma_{k,1} = \frac{1}{(k+3)^{3/5}},\quad \gamma_{k,2} = \frac{1}{2(k+3)^{3/5}}\qquad\forall k\ge0,\label{eta-gamma-3}\\
&\theta_{k,1}= \frac{2(k+3)^{3/5}-1}{(k+3)^{6/5}},\quad \theta_{k,2}= \frac{1-(k+3)^{3/5}}{2(k+3)^{6/5}}\qquad\forall k\ge0.\label{theta-3}
\end{align}
Let $\iota_K$ be uniformly drawn from $\{0,\ldots,K-1\}$. Then,
\begin{align}\label{expect-kappa-3}
\E[\|\nabla f(x^{\iota_K})\|]\le  \epsilon\qquad\forall K\ge \max\Big\{\Big(\frac{20M_3}{3\epsilon}\ln\Big(\frac{20M_3}{3\epsilon}\Big)\Big)^{10/3},5\Big\},
\end{align}
where $\epsilon\in(0,1)$, and the expectation is taken over the randomness in \cref{alg:unf-sfom} and in the drawing of $\iota_K$.
\end{theorem}

\begin{remark}
From \cref{cor:rate-3}, we observe that \cref{alg:unf-sfom} achieves a sample and first-order operation complexity of $\widetilde{\mathcal{O}}(\epsilon^{-10/3})$ for finding an $\epsilon$-stochastic stationary point of problem \cref{pb:uc} under \cref{asp:basic} and \cref{asp:pth-smth} with $p=3$.   
\end{remark}

\subsection{An SFOM with multi-extrapolated momentum}\label{subsec:sfom-mem}

In this subsection, we analyze \cref{alg:unf-sfom} with $q=p-1$, which is capable of exploiting the smoothness of $D^p f$ for some $p\ge2$. In the following, we establish the sample and first-order operation complexity of \cref{alg:unf-sfom} with $q=p-1$ under Assumption \ref{asp:basic} and Assumption \ref{asp:pth-smth} for some $p\ge2$.

Throughout this subsection, we impose the following systems of equations on the input parameters $\{(\gamma_{k,t},\theta_{k,t})\}_{1\le t\le q,k\ge0}$ of \cref{alg:unf-sfom}:
\begin{align}\label{pth-lss}
\begin{bmatrix}
1/\gamma_{k,1} & 1/\gamma_{k,2} & \cdots & 1/\gamma_{k,q}\\ 
1/\gamma_{k,1}^2 & 1/\gamma_{k,2}^2 & \cdots & 1/\gamma_{k,q}^2\\ 
\vdots&\vdots&\ddots & \vdots\\
1/\gamma_{k,1}^q & 1/\gamma_{k,2}^q & \cdots & 1/\gamma_{k,q}^q
\end{bmatrix}\begin{bmatrix}
\theta_{k,1}\\
\theta_{k,2}\\
\vdots\\
\theta_{k,q}
\end{bmatrix}  = \begin{bmatrix}
1\\
1\\
\vdots\\
1
\end{bmatrix}\qquad\forall k\ge0.
\end{align}
In addition, we require that 
\begin{align}\label{pth-theta-01}
\sum_{t=1}^q\theta_{k,t} \in(0,1)\qquad\forall k\ge0.
\end{align}
Given $\{\gamma_{k,t}\}_{1\le t\le q}$, \cref{pth-lss} can be viewed as a system of linear equations with unknowns $\{\theta_{k,t}\}_{1\le t\le q}$, where the coefficient matrix in \cref{pth-lss} takes the form of a Vandermonde matrix (e.g., see \cite{horn2012matrix,humpherys2020foundations}). %It can be verified that as long as the values of $\gamma_{k,t}$, $1\le t\le q$, are positive and distinct, this matrix is invertible, and therefore the the values of $\theta_{k,t}$, $1\le t\le q$, can be uniquely determined by solving \cref{pth-lss}. 
The following lemma demonstrates that if the values of $\{\gamma_{k,t}\}_{1\le t\le q}$ are positive and distinct, the values of $\{\theta_{k,t}\}_{1\le t\le q}$ can be uniquely determined by solving \cref{pth-lss}. In addition, this lemma provides an explicit expression for the solution to \cref{pth-lss}, along with the alternating-sign property of $\{\theta_{k,t}\}_{1\le t\le q}$, assuming that the values of $\{\gamma_{k,t}\}_{1\le t\le q}$ are ordered. Its proof is deferred to \cref{subsec:proof-mem}.

\begin{lemma}\label{lem:p-sol}
Assume that $\{\gamma_{k,t}\}_{1\le t\le q}\subset(0,1)$ with $\gamma_{k,1}>\cdots>\gamma_{k,q}$ are given for some $k\ge0$. Then, the solution $\{\theta_{k,t}\}_{1\le t\le q}$ to the linear system in \cref{pth-lss} is unique and can be explicitly written as
\begin{align}\label{theta-kt-explicit}
\theta_{k,t} = \frac{\prod_{1\le s\le q,s\neq t}(1-1/\gamma_{k,s})}{1/\gamma_{k,t}\prod_{1\le s\le q,s\neq t}(1/\gamma_{k,t}-1/\gamma_{k,s})}\qquad\forall 1\le t\le q,    
\end{align} 
which satisfies $\theta_{k,t}>0$ for all odd $t$ and $\theta_{k,t}<0$ for all even $t$. Moreover, it holds that
\begin{align}\label{sum-q-theta}
\sum_{t=1}^q   \theta_{k,t} = 1 - \frac{\prod_{t=1}^q(1/\gamma_{k,t}-1)}{\prod_{t=1}^q1/\gamma_{k,t}}. 
\end{align}
\end{lemma}

The specific selection of $\{(\gamma_{k,t},\theta_{k,t})\}_{1\le t\le q,k\ge0}$ for \cref{alg:unf-sfom} such that \cref{pth-lss} and \cref{pth-theta-01} hold is given in \cref{eta-gamma-p} and \cref{theta-p} below. We next provide some useful recurrence properties for \cref{alg:unf-sfom} assuming that \cref{pth-lss} and \cref{pth-theta-01} hold. The following lemma establishes the recurrence for the estimation error of the gradient estimators $\{m^k\}_{k\ge0}$ of \cref{alg:unf-sfom} with $q=p-1$. Its proof is deferred to \cref{subsec:proof-mem}.

%In what follows, we analyze \cref{alg:unf-sfom} with $q=p-1$ and under the assumption that input parameters $\{(\gamma_{k,t},\theta_{k,t})\}_{1\le t\le p-1, k\ge0}$ satisfy \cref{pth-lss} and \cref{pth-theta-01}. The specific selection of $\{(\gamma_{k,t},\theta_{k,t})\}_{1\le t\le p-1, k\ge0}$ that satisfy \cref{pth-lss} and \cref{pth-theta-01} are provided later in \cref{eta-gamma-p} and \cref{theta-p}. 

\begin{lemma}\label{lem:pth-rec}
Suppose that \cref{asp:basic} holds, and \cref{asp:pth-smth} holds for some $p\ge2$. Let $\{(x^k,m^k)\}_{k\ge0}$ be generated by \cref{alg:unf-sfom} with inputs $q=p-1$, step sizes $\{\eta_k\}_{k\ge0}$, and $\{(\gamma_{k,t},\theta_{k,t})\}_{1\le t\le p-1, k\ge0}$ satisfying \cref{pth-lss} and \cref{pth-theta-01}. Then, 
\begin{align}\label{rec:pth-m}
&\mathbb{E}_{\xi^{k+1}}[\|m^{k+1} - \nabla f (x^{k+1})\|^2]\nonumber\\
&\le \Big(1-\sum_{t=1}^{p-1}\theta_{k,t}\Big)\|m^k - \nabla f (x^k)\|^2 +  \frac{pL_p^2\eta_k^{2p}}{(p!)^2\sum_{t=1}^{p-1}\theta_{k,t}} \Big(1+\sum_{t=1}^{p-1} \frac{\theta_{k,t}^2}{\gamma_{k,t}^{2p}}\Big) + (p-1)\sigma^2\sum_{t=1}^{p-1}\theta_{k,t}^2\qquad\forall k\ge0,
\end{align}    
where $\sigma$ is given in Assumption \ref{asp:basic}(b), and $p$ and $L_p$ are given in Assumption \ref{asp:pth-smth}.
\end{lemma}

The following theorem establishes the descent relation of potentials $\{P_k\}_{k\ge0}$ for \cref{alg:unf-sfom} with $q=p-1$, whose proof is deferred to \cref{subsec:proof-mem}.

%We next derive an upper bound for the average expected error of the stationary condition among all iterates generated by \cref{alg:unf-sfom} with $q=p-1$. Its proof is relegated to \cref{subsec:proof-mem}.

\begin{theorem}\label{thm:pth-conv}
Suppose that \cref{asp:basic} holds, and \cref{asp:pth-smth} holds for some $p\ge2$. Let $\{(x^k,m^k)\}_{k\ge0}$ be generated by \cref{alg:unf-sfom} with inputs $q=p-1$, step sizes $\{\eta_k\}_{k\ge0}$, and $\{(\gamma_{k,t},\theta_{k,t})\}_{1\le t\le p-1, k\ge0}$ satisfying \cref{pth-lss} and \cref{pth-theta-01}. Define
\begin{align}\label{def:pth-Pk}
P_k \doteq f(x^k) + p_k\|m^k - \nabla f(x^k)\|^2\qquad\forall k\ge0,
\end{align}
where $\{p_k\}_{k\ge0}$ is a positive sequence such that
\begin{align}\label{thetak-pth}
\Big(1- \sum_{t=1}^{p-1}\theta_{k,t}\Big)p_{k+1}\le \Big(1-\sum_{t=1}^{p-1}\theta_{k,t}/(p+1)\Big)p_{k}\qquad\forall k\ge0.
\end{align}
Then, 
\begin{align}
&\E_{\xi^{k+1}}[P_{k+1}]\le P_k - \eta_k\|\nabla f(x^k)\| + \frac{L_1}{2}\eta_k^2\nonumber\\
&\qquad\qquad + \frac{(p+1)\eta_k^2}{p_k\sum_{t=1}^{p-1}\theta_{k,t}} +  \frac{pL_p^2\eta_k^{2p}p_{k+1}}{(p!)^2\sum_{t=1}^{p-1}\theta_{k,t}} \Big(1 +\sum_{t=1}^{p-1} \frac{\theta_{k,t}^2}{\gamma_{k,t}^{2p}}\Big) + (p-1)\sigma^2p_{k+1}\sum_{t=1}^{p-1}\theta_{k,t}^2\qquad\forall k\ge0,\label{pth-stat-upbd} 
\end{align}
where $f_{\mathrm{low}}$, $L_1$, and $\sigma$ are given in \cref{asp:basic}, and $p$ and $L_p$ are given in \cref{asp:pth-smth}.
\end{theorem}

%It then can be verified that
%\begin{align*}
%\sum_{t=1}^{p-1}\theta_{k,t}=\prod_{1\le t\le q}(t(k)^{2p/(3p+1)}-1)\sum_{t=1}^{p-1}\frac{1}{(i(k)^{2p/(3p+1)}-1)(k)^{4p^2/(3p+1)}i\prod_{t\neq i, 1\le t\le q}(t-i)}    
%\end{align*}

The following theorem establishes an upper bound for the number of iterations of \cref{alg:unf-sfom} with $q=p-1$ before finding an $\epsilon$-stochastic stationary point of problem \cref{pb:uc}. Its proof is deferred to \cref{subsec:proof-mem}.

\begin{theorem}\label{cor:rate-p}
Suppose that \cref{asp:basic} holds, and \cref{asp:pth-smth} holds for some $p\ge2$. Define 
\begin{align}\label{M-p}
M_p\doteq 4\Big(f(x^0) - f_{\mathrm{low}} + p\sigma^2 +\frac{3L_1}{2} + \frac{7L_p^2}{(p!)^2} + 2(p+1 + 32p^{2p}L_p^2 + 16(p!)^2\sigma^2) \Big),
\end{align}
where $f_{\mathrm{low}}$, $L_1$, and $\sigma$ are given in \cref{asp:basic}, and $p$ and $L_p$ are given in \cref{asp:pth-smth}. Let $\{x^k\}_{k\ge0}$ be all iterates generated by \cref{alg:unf-sfom} with inputs $q=p-1$ and $\{(\eta_k,\gamma_{k,t},\theta_{k,t})\}_{1\le t\le p-1,k\ge0}$ specified as
\begin{align}
&\eta_k = \frac{1}{(k+p)^{(2p+1)/(3p+1)}},\quad \gamma_{k,t} = \frac{1}{t(k+p)^{2p/(3p+1)}}\qquad \forall 1\le t\le p-1,k\ge0,\label{eta-gamma-p}\\
&\theta_{k,t} = \frac{\prod_{1\le s\le p-1,s\neq t}(1-s(k+p)^{2p/(3p+1)})}{t(k+p)^{2p/(3p+1)}\prod_{1\le s\le p-1,s\neq t}((t-s)(k+p)^{2p/(3p+1)})}\qquad\forall 1\le t\le p-1,k\ge0.\label{theta-p}
\end{align}
Let $\iota_K$ be uniformly drawn from $\{0,\ldots,K-1\}$. Then,
\begin{align}\label{expect-kappa-p}
\E[\|\nabla f(x^{\iota_K})\|]\le \epsilon\qquad \forall k\ge \max\Big\{ \Big(\frac{(6p+2)M_p}{p\epsilon}\ln\Big(\frac{(6p+2)M_p}{p\epsilon}\Big)\Big)^{(3p+1)/p}, 2p\Big\},
\end{align}
where $\epsilon\in(0,1)$, and the expectation is taken over the randomness in \cref{alg:unf-sfom} and in the drawing of $\iota_K$.
\end{theorem}

\begin{remark}
From \cref{cor:rate-p}, we observe that \cref{alg:unf-sfom} achieves a sample and first-order operation complexity of $\widetilde{\mathcal{O}}(\epsilon^{-(3p+1)/p})$ for finding an $\epsilon$-stochastic stationary point of problem \cref{pb:uc} under \cref{asp:basic} and \cref{asp:pth-smth} with some $p\ge2$. This result generalizes the complexity of $\mathcal{O}(\epsilon^{-4})$ under the assumption of Lipschitz continuous gradients \cite{cutkosky2020momentum, ghadimi2013stochastic, ghadimi2016accelerated}, and the complexity of $\mathcal{O}(\epsilon^{-7/2})$ under the assumption of Lipschitz continuous Hessians \cite{cutkosky2020momentum, fang2019sharp}.
\end{remark}

\section{Numerical experiments}\label{sec:ne}
In this section, we conduct preliminary numerical experiments to evaluate the practical performance of our proposed SFOMs with multi-extrapolated momentum (\cref{alg:unf-sfom}). We compare our proposed method with the normalized stochastic gradient method with Polyak momentum (abbreviated as SG-PM), studied in \cite{cutkosky2020momentum}, for solving a data fitting problem and a robust regression problem with noisy gradient evaluations. All algorithms are implemented in MATLAB, and the computations are performed on a laptop with a 2.20 GHz Intel Core i9-14900HX processor and 32 GB of RAM.

\subsection{Data fitting problems}\label{subsec:df-pb}
In this subsection, we consider a data fitting problem:
\begin{align}\label{df}
\min_{x\in\R^n}  \Big\{f(x) = \sum_{i=1}^n \|s(a_i^Tx)-b_i\|^2\Big\},
\end{align}
where $s(t)=e^t/(1+e^t)$ is the sigmoid function, and $\{(a_i,b_i)\}_{1\le i\le n}\subset\R^n\times\R$ denotes the given data. It can be verified that the objective function satisfies Assumptions \ref{asp:basic}(a), \ref{asp:basic}(b), and \cref{asp:pth-smth} with $p=2,3$. We stimulate the noisy gradient evaluations by setting $G(x;\xi)=\nabla f(x) +  \Tilde{\sigma}\xi\min\{\sqrt{\|x\|},1\}\mathbf{1}$ for all $x\in\mathbb{R}^n$, where $\mathbf{1}\in\R^n$ is the all-one vector, $\xi$ follows the multivariate standard normal distribution, and $\Tilde{\sigma}>0$ is a given constant.  Then, in view of \cref{pro:special}, we see that $G(\cdot;\cdot)$ satisfies \cref{asp:basic}(c) but does not satisfy the mean-squared smoothness condition.

\begin{figure}[ht]
\centering
\begin{minipage}[b]{0.32\linewidth}
\centering
\includegraphics[width=\linewidth]{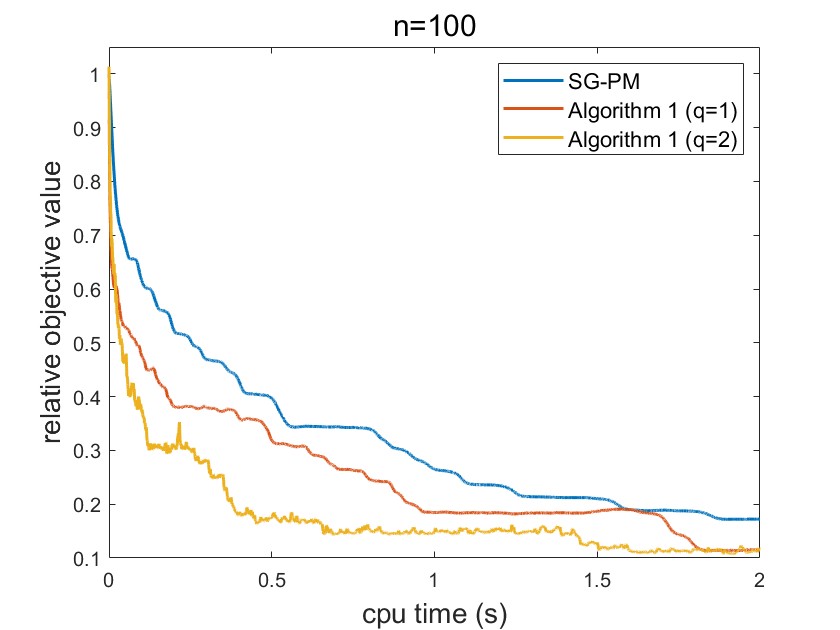}
\end{minipage}
\begin{minipage}[b]{0.32\linewidth}
\centering
\hfill\includegraphics[width=\linewidth]{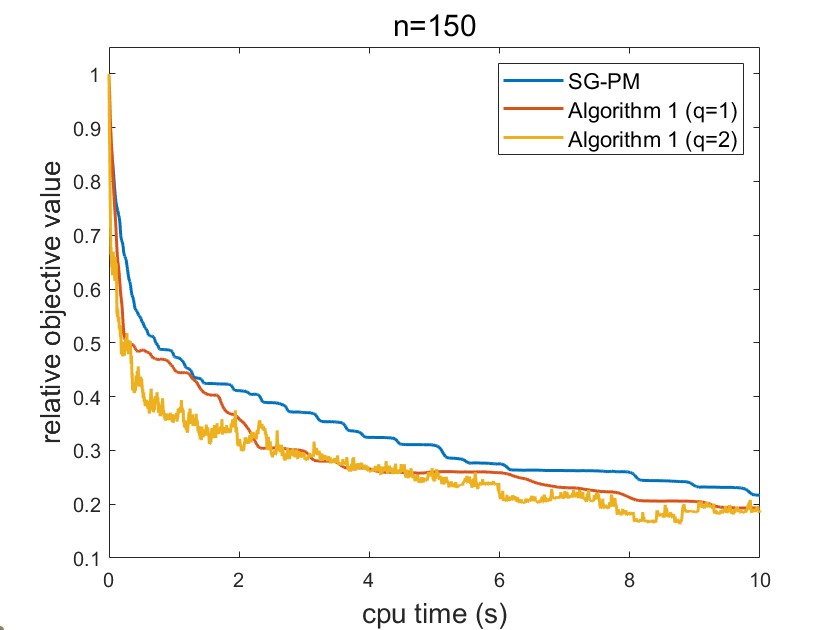}
\end{minipage}
\begin{minipage}[b]{0.32\linewidth}
\centering
\hfill\includegraphics[width=\linewidth]{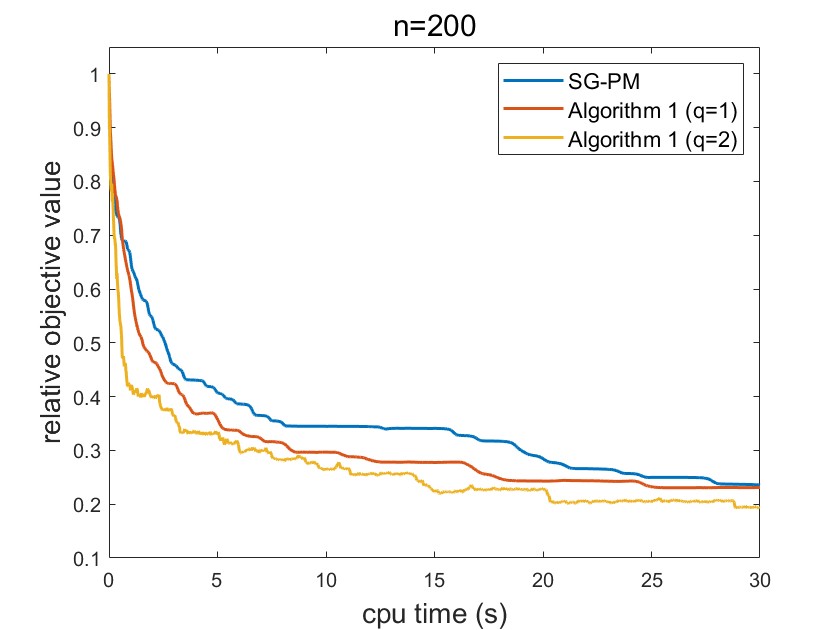}
\end{minipage}
\caption{Convergence behavior of the relative objective value for all SFOMs in solving problem~\cref{df}.}
\label{fig:alg1}
\end{figure}

For $n=100$, $150$, and $200$, we randomly generate $a_i$, $1\le i\le n$, with all entries independently chosen from the standard normal distribution. We also randomly generate the true solution $x^*$ with all entries independently chosen from the standard normal distribution. We then set $b_i$ to be $b_i=s(a_i^Tx^*)+0.1e_i$ for each $1\le i\le n$, where $e_i$, $1\le i\le n$, are independently chosen from the standard normal distribution. We set $\Tilde{\sigma}=10$ when simulating the noisy gradient evaluations.

We apply our \cref{alg:unf-sfom} with $q=1,2$, as well as SG-PM to solve \cref{df}. We compare these methods in terms of relative objective value, defined as $f(x^k)/f(x^0)$. For $n=100$, $150$, and $200$, we set the maximum CPU time to $2$, $10$, and $30$ seconds, respectively, and the initial iterate as the all-one vector. We tune the other input parameters of all methods to optimize their individual practical performance.

For each choice of $n$, we plot the relative objective value in \cref{fig:alg1} to illustrate the convergence behavior of all SFOMs for solving problem~\cref{df}. From this figure, we observe that \cref{alg:unf-sfom} with $q=1,2$ is faster than SG-PM, and moreover, \cref{alg:unf-sfom} with $q=2$ outperforms \cref{alg:unf-sfom} with $q=1$. These observations support our theoretical results.

\subsection{Robust regression problems}
In this subsection, we consider a robust regression problem:
\begin{align}\label{robust-reg}
\min_{x\in\R^n}  \Big\{f(x) = \sum_{i=1}^m \phi(a_i^Tx-b_i)\Big\},
\end{align}
where $\phi(t)=t^2/(1+t^2)$ is a robust loss function \cite{carmon2017convex,he2023newton}, and $\{(a_i,b_i)\}_{1\le i\le m}\subset\R^n\times\R$ is the training set. It can be verified that the objective function satisfies Assumptions \ref{asp:basic}(a), \ref{asp:basic}(b), and \cref{asp:pth-smth} with $p=2,3$. Similar to \cref{subsec:df-pb}, we stimulate the noisy gradient evaluations by setting $G(x;\xi)=\nabla f(x) +  \Tilde{\sigma}\xi\min\{\sqrt{\|x\|},1\}\mathbf{1}$ for all $x\in\mathbb{R}^n$, where $\mathbf{1}\in\R^n$ denotes the all-one vector, $\xi$ follows the standard normal distribution, and $\Tilde{\sigma}>0$ is a given constant.  Then, in view of \cref{pro:special}, we see that $G(\cdot;\cdot)$ satisfies \cref{asp:basic}(c) but does not satisfy the mean-squared smoothness condition.

We consider three real datasets, `red wine quality', `white wine quality', and `real estate valuation' from 
the UCI repository.\footnote{see \url{archive.ics.uci.edu/datasets}} We rescale the features and predictions so that their values lie in $[0,1]$. We set $\Tilde{\sigma}=100$ when simulating the noisy gradient evaluations. 

We apply our \cref{alg:unf-sfom} with $q=1,2$, as well as SG-PM to solve \cref{robust-reg}. We compare these methods in terms of relative loss, defined as $f(x^k)/f(x^0)$. For each method, we set the maximum CPU time to 5 seconds and the initial iterate as the all-one vector. We tune the other input parameters of all methods to optimize their individual practical performance.

\begin{figure}[ht]
\centering
\begin{minipage}[b]{0.32\linewidth}
\centering
\includegraphics[width=\linewidth]{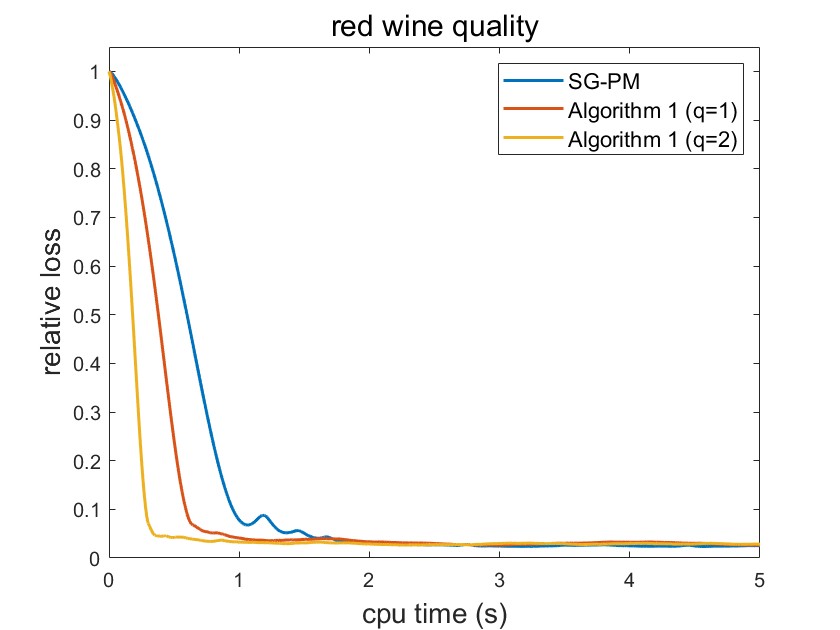}
\end{minipage}
\begin{minipage}[b]{0.32\linewidth}
\centering
\hfill\includegraphics[width=\linewidth]{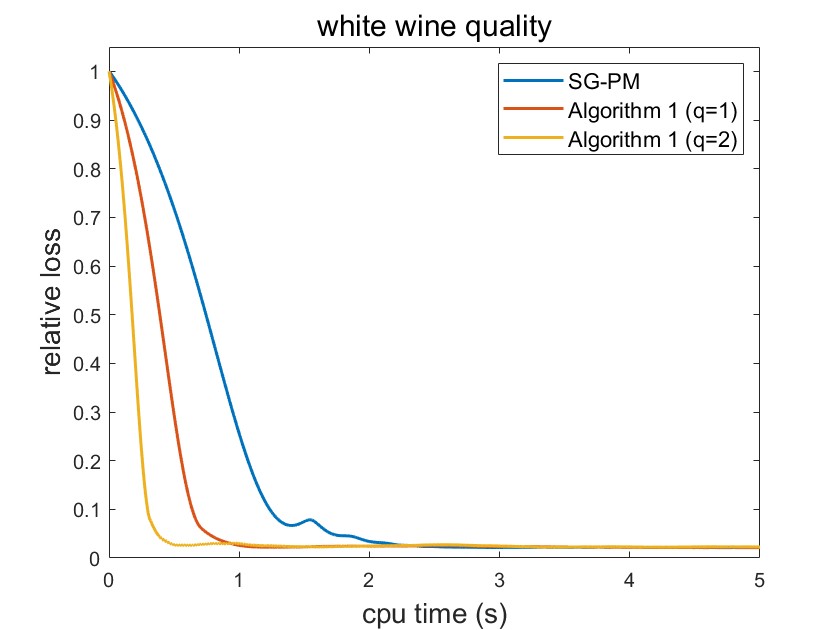}
\end{minipage}
\begin{minipage}[b]{0.32\linewidth}
\centering
\hfill\includegraphics[width=\linewidth]{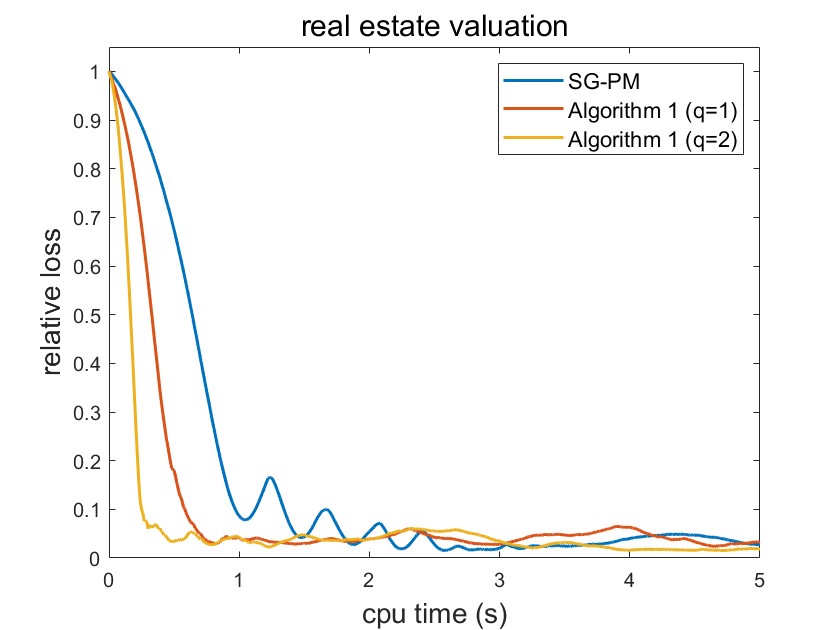}
\end{minipage}
\caption{Convergence behavior of the relative loss for all SFOMs in solving problem~\cref{robust-reg}.}
\label{fig:alg}
\end{figure}

For each dataset, we plot the relative loss in \cref{fig:alg} to illustrate the convergence behavior of all SFOMs when solving problem~\cref{robust-reg}. From \cref{fig:alg}, we observe that \cref{alg:unf-sfom} with $q=1,2$ outperforms SG-PM. In addition, \cref{alg:unf-sfom} with $q=2$ outperforms \cref{alg:unf-sfom} with $q=1$. These observations corroborate our theoretical results.

%with $q=2,3$ slightly outperforms SG-PM and comes very close to the performance of STORM. This corroborates our theoretical finding that more extrapolations can achieve faster convergence. In addition, we notice that \cref{alg:unf-sfom} with $q=1$, which is similar to NIGT \cite{cutkosky2020momentum}, is slightly worse than other competing methods.

%Compared to the case without extrapolation, it achieves acceleration and comes very close to the convergence speed of STORM.

\section{Proof of the main results}\label{sec:proof}

In this section, we provide proofs of the main results in \cref{sec:not-asm,sec:method}, which are particularly \cref{lem:pth-desc,lem:1-desc,lem:3rd-rec,lem:p-sol,lem:pth-rec} and \cref{thm:3rd-conv,thm:pth-conv,cor:rate-3,cor:rate-p}.

To proceed, we first establish several technical lemmas. The following lemma concerns the estimation of the partial sums of series.

\begin{lemma}\label{lem:series}
Let $\zeta(\cdot)$ be a convex univariate function. Then, for any integers $a,b$ satisfying $[a-1/2,b+1/2]\subset\mathrm{dom}\zeta$, it holds that $\sum_{p=a}^b\zeta(p)\le \int^{b+1/2}_{a-1/2}\zeta(\tau)\mathrm{d}\tau$.
\end{lemma}

\begin{proof}
Since $f$ is convex, one has $\zeta(p)\le \int_{p-1/2}^{p+1/2}\zeta(\tau)\mathrm{d}\tau$ for all $p\in[a,b]$. It then follows that $\sum_{p=a}^b\zeta(p)\le\int_{a-1/2}^{b+1/2}\zeta(\tau)\mathrm{d}\tau$ holds as desired.
\end{proof}

As a consequence of \cref{lem:series}, we consider $\zeta(\tau)=1/\tau^\alpha$ for $\alpha\in(0,\infty)$, where $\tau\in(0,\infty)$. Then, for any positive integers $a,b$, one has
\begin{align}
\sum_{p=a}^b \frac{1}{p^\alpha} \le \left\{\begin{array}{ll}
\ln\big(b+\frac{1}{2}\big) - \ln\big(a-\frac{1}{2}\big)&\text{if }\alpha=1,  \\[4pt]
\frac{1}{1-\alpha}\big(\big(b+\frac{1}{2}\big)^{1-\alpha} - \big(a-\frac{1}{2}\big)^{1-\alpha}\big)&\text{if }\alpha\in(0,1)\cup(1,+\infty).
\end{array}\right.
\label{upbd:series-ka}
\end{align}

We next provide an auxiliary lemma that will be used to analyze the complexity of our algorithm.

\begin{lemma}\label{lem:rate-complexity}
Let $\alpha\in(0,1)$ and $u\in(0,1/e)$ be given. Then, $1/v^\alpha\ln v\le 2u/\alpha$ holds for all $v$ satisfying $v\ge(1/u\ln(1/u))^{1/\alpha}$.
\end{lemma}

\begin{proof}
Let $v \in \mathbb{R}$ be arbitrarily chosen such that $v\ge(1/u\log(1/u))^{1/\alpha}$. Then, we notice from $u\in(0,1/e)$ that
\begin{align}\label{aux-uv-ln}
v\ge(1/u\log(1/u))^{1/\alpha}> e^{1/\alpha}.    
\end{align}
Denote $\phi(v)\doteq 1/v^{\alpha}\log v$. Since $\phi$ is decreasing over $(e^{1/\alpha},\infty)$, it follows from \cref{aux-uv-ln} that 
\begin{align*}
1/v^{\alpha}\log v = \phi(v) \le  \phi((1/u\log(1/u))^{1/\alpha}) = \frac{u}{\alpha} \Big(1+\frac{\log\log(1/u)}{\log(1/u)}\Big)\le \frac{2u}{\alpha},
\end{align*}
where the last inequality is due to $\log\log(1/u)\le\log(1/u)$ for all $u\in(0,1/e)$. Hence, the conclusion of this lemma holds as desired.
\end{proof}

The following lemma concerns the uniqueness of polynomial interpolation. It can also be found in \cite[Section 9.2]{humpherys2020foundations}.
\begin{lemma}\label{lem:lag-inter}
For any $m+1$ pairs $\{(a_i,b_i)\}_{0\le i\le m}\subset\R\times\R$, where the values of $\{a_i\}_{0\le i\le m}$ are distinct, there exists a unique polynomial $\psi(\alpha)$ of degree at most $m$ such that $\psi(a_i)=b_i$ for all $0\le i\le m$. Moreover, such polynomial can be written as follows:
\begin{align}
\psi(\alpha) = \sum_{i=0}^m\frac{\prod_{0\le j\le m,j\neq i}(\alpha-a_j)}{\prod_{0\le j\le m,j\neq i}(a_i-a_j)}b_i.
\end{align}
\end{lemma}
\begin{proof}
One can verify that $\psi(a_i) = b_i$ for all $0 \leq i \leq m$. Next, we prove the uniqueness of $\psi$. Suppose there exist two polynomials, $\psi_1$ and $\psi_2$, both of degree at most $m$, and both satisfy $\psi_1(a_i) = \psi_2(a_i) = b_i$ for all $0 \leq i \leq m$. It follows that the polynomial $\psi_0(\alpha) \doteq \psi_1(\alpha) - \psi_2(\alpha)$ has degree at most $m$ and satisfies $\psi_0(a_i) = 0$ for all $0 \leq i \leq m$, where the values of $a_i$'s, $0 \leq i \leq m$, are distinct. Therefore, we can conclude that the polynomial has degree at most $m$ and has $m+1$ distinct zeros, which implies that $\psi_0(\alpha) \equiv 0$ for all $\alpha \in \mathbb{R}$. Hence, $\psi_1(\alpha)\equiv \psi_2(\alpha)$ for all $\alpha \in \mathbb{R}$, which proves the uniqueness of $\psi$.
\end{proof}

\subsection{Proof of \cref{lem:pth-desc,lem:1-desc}}\label{sec:proof-lem12}

\begin{proof}[Proof of \cref{lem:pth-desc}]
Fix an arbitrary $u\in\R^n$. We denote $\phi(x) \doteq \langle \nabla f(x),u\rangle$. Using this and the definition of $\nabla^{r+1} f(x)(h)^{r}$, one has that
\begin{align}\label{pth-der-phi-f}
D^r\phi(x)[v]^r = \langle \nabla^{r+1} f(x)(v)^r, u\rangle\qquad\forall 1\le r\le p-1, v\in\R^n.    
\end{align}
Since $f$ is $p$th-order continuously differentiable, it follows that $\phi$ is $(p-1)$th-order continuously differentiable, and that
\begin{align}\label{p-p-1th-der-phi-f}
\|D^{p-1}\phi(y)-D^{p-1}\phi(x)\|_{(p-1)} = \|u\|\|D^p f(y)-D^p f(x)\|_{(p)}\qquad \forall x,y\in\R^n.
\end{align}
Fix any $x,y\in\R^n$. Using Taylor's theorem, we have 
\begin{align}
&\phi(y) = \phi(x) + \sum_{r=1}^{p-2}\frac{1}{r!} D^r\phi(x)[y-x]^r  + \frac{1}{(p-2)!} \int_0^1(1-t)^{p-2}D^{p-1} \phi(x+t(y-x))[y-x]^{p-1}\mathrm{d}t\nonumber\\
& = \phi(x) + \sum_{r=1}^{p-1}\frac{1}{r!} D^r\phi(x)[y-x]^r \nonumber\\
&\qquad + \frac{1}{(p-2)!} \int_0^1(1-t)^{p-2}(D^{p-1} \phi(x+t(y-x)) - D^{p-1} \phi(x))[y-x]^{p-1}\mathrm{d}t,\label{p-phi-taylor}
\end{align}
where the second equation is due to $\int^1_0(1-t)^{p-2}\mathrm{d}t=1/(p-1)$. Combining these, we obtain that 
\begin{align*}
&\Big|\Big\langle \nabla f(y) - \nabla f(x) - \sum_{r=1}^{p-1}\frac{1}{r!} \nabla^{r+1}f(x)(y-x)^r, u\Big\rangle \Big| \overset{\cref{pth-der-phi-f}}{=} \Big|\phi(y) - \phi(x) - \sum_{r=1}^{p-1}\frac{1}{r!} D^r\phi(x)[y-x]^r \Big|\\
&\overset{\cref{p-phi-taylor}}{=}\Big|\frac{1}{(p-2)!} \int_0^1(1-t)^{p-2}(D^{p-1} \phi(x+t(y-x)) - D^{p-1} \phi(x))[y-x]^{p-1}\mathrm{d}t\Big|\\
&\overset{\cref{def:pnorm}}{\le} \frac{1}{(p-2)!}\|y-x\|^{p-1}\int^1_0(1-t)^{p-2}\|D^{p-1} \phi(x+t(y-x)) - D^{p-1} \phi(x)\|_{(p-1)}\mathrm{d}t\\
&\overset{\cref{p-p-1th-der-phi-f}}{=} \frac{1}{(p-2)!}\|y-x\|^{p-1}\|u\|\int^1_0(1-t)^{p-2}\|D^p f(x+t(y-x)) - D^pf(x)\|_{(p)}\mathrm{d}t\\
&{\le} \frac{1}{(p-2)!}L_p\|y-x\|^p\|u\|\int^1_0(1-t)^{p-2}t\mathrm{d}t = \frac{1}{p!}L_p\|y-x\|^p\|u\|,
\end{align*}
where the last inequality is due to \cref{asp:pth-smth}, and the last equality is due to $\int^1_0(1-t)^{p-2}t\mathrm{d}t=1/(p(p-1))$. Taking the maximum of this inequality over all $u$ with $\|u\|\le 1$, we obtain that \cref{ineq:pth-desc} holds.
\end{proof}

\begin{proof}[Proof of \cref{lem:1-desc}]
Fix any $k\ge0$. Using \cref{ineq:1st-desc} with $(x,y)=(x^k,x^{k+1})$, we obtain that
\begin{align*}
&f(x^{k+1}) \overset{\cref{ineq:1st-desc}}{\le} f(x^k) + \langle\nabla f(x^k), x^{k+1} - x^k \rangle + \frac{L_1}{2}\|x^{k+1}-x^k\|^2\\
&= f(x^k) + \langle m^k, x^{k+1} - x^k\rangle + \langle \nabla f(x^k) - m^k, x^{k+1} - x^k\rangle + \frac{L_1}{2}\|x^{k+1}-x^k\|^2\\
&\overset{\cref{update-xk}}{=} f(x^k) -\eta_k \|m^k\| -\frac{\eta_k}{\|m^k\|} \langle \nabla f(x^k) - m^k, m^k\rangle + \frac{L_1}{2}\eta_k^2\\
&\le f(x^k) - \eta_k\|m^k\| + \eta_k\|\nabla f(x^k) - m^k\| + \frac{L_1}{2}\eta_k^2\\
&\le f(x^k) - \eta_k\|\nabla f(x^k)\| + 2\eta_k\|\nabla f(x^k) - m^k\| + \frac{L_1}{2}\eta_k^2,
\end{align*}
where the second inequality follows from the Cauchy-Schwarz inequality, and the last inequality is due to the triangular inequality. Hence, \cref{lem:1-desc} holds as desired.
\end{proof}

%The inverse of a Vandermonde matrix.

%\begin{lemma}
%Let $0<a_1 < a_2 < \cdots < a_q$, and define a Vandermonde %matrix as
%\begin{align}
%V(a_1,\ldots,a_n)=\begin{bmatrix}
%a_1 & a_1^2 & \cdots & a_1^q \\
%a_2 & a_2^2 & \cdots & a_2^q \\
%\vdots & \vdots & \ddots & \vdots \\
%a_q & a_q^2 & \cdots & a_q^q \\
%\end{bmatrix}.    
%\end{align}
%Then, $V(a_1,\ldots,a_n)^{-1}=[u_{ij}]_{q\times q}$, where the %coefficients $u_{ij}$, $1\le i,j\le q$, can be determined by the polynomial equation
%\[
%\sum_{i=1}^q u_{ij}\alpha^i \equiv \frac{\alpha\prod_{t\neq j, 1\le t\le q} (\alpha - a_t)}{a_j\prod_{t\neq j, 1\le t\le q} (a_j - a_t)}\qquad \forall \alpha\in\R.
%\]

%\end{lemma}

%\begin{proof}
%Notice that
%\begin{align}
%V^Tc=\begin{bmatrix}
%a & a^2 & \cdots & a^q \\
%2a & (2a)^2 & \cdots & (2a)^q \\
%\vdots & \vdots & \ddots & \vdots \\
%qa & (qa)^2 & \cdots & (qa)^q \\
%\end{bmatrix}\begin{bmatrix}
%c_1 \\
%c_2 \\
%\vdots \\
%c_q \\
%\end{bmatrix} = \begin{bmatrix}
%h(a)\\
%h(2a)\\
%\vdots\\
%h(qa)
%\end{bmatrix},    
%\end{align}
%where
%\[
%h(\alpha) = \sum_{i=1}^q c_i \alpha^i.
%\]
%Then for any $y\in\R^q$, we can solve
%\[
%V^Tc = \begin{bmatrix}
%h(a)\\
%h(2a)\\
%\vdots\\
%h(qa)
%\end{bmatrix} = \begin{bmatrix}
%    y_1\\
%    y_2\\
%    \vdots\\
%    y_q
%\end{bmatrix} = y.
%\]
%Applying this, we have
%\[
%h(\alpha) = y_1 \frac{\prod_{s\neq 1} (\alpha - sa)}{\prod_{s\neq 1} (a - sa)} + y_2 \frac{\prod_{s\neq 2} (\alpha - sa)}{\prod_{s\neq 2} (2a - sa)} + \cdots + y_q \frac{\prod_{s\neq q} (\alpha - sa)}{\prod_{s\neq q} (qa - sa)}.
%\]

%\end{proof}

\subsection{Proof of the main results in \cref{subsec:sfom-dem}}\label{subsec:proof-dem}

In this subsection, we prove \cref{lem:3rd-rec,thm:3rd-conv,cor:rate-3}. 

When \cref{asp:pth-smth} holds with $p=3$, it directly follows from \cref{lem:pth-desc} that
\begin{align}\label{ineq:3rd-desc}
\Big\|\nabla f(y) - \nabla f(x) - \nabla^2 f(x)(y-x) - \frac{1}{2}\nabla^3 f(x)(y-x)^2\Big\|\le \frac{L_3}{6}\|y-x\|^3\qquad\forall x,y\in\R^n.
\end{align}   

\begin{proof}[{Proof of \cref{lem:3rd-rec}}]
Fix any $k\ge0$. Notice from \cref{update-zk} with $q=2$ that
\begin{align}\label{3-z-x}
z^{k+1,1} - x^k = \frac{1}{\gamma_{k,1}}(x^{k+1} - x^k),\quad     z^{k+1,2} - x^k = \frac{1}{\gamma_{k,2}}(x^{k+1} - x^k).
\end{align}
By this, \cref{3rd-lss}, and \cref{update-mk}, one has that 
\begin{align*}
&m^{k+1} - \nabla f (x^{k+1}) \overset{ \cref{update-mk}}{=} (1 - \theta_{k,1} - \theta_{k,2})m^k + \theta_{k,1}G(z^{k+1,1};\xi^{k+1}) + \theta_{k,2}G(z^{k+1,2};\xi^{k+1}) - \nabla f (x^{k+1}) \\
&\overset{\cref{3rd-lss}}{=} (1 - \theta_{k,1} - \theta_{k,2})m^k + \theta_{k,1}G(z^{k+1,1};\xi^{k+1}) + \theta_{k,2}G(z^{k+1,2};\xi^{k+1}) - \nabla f (x^{k+1})\\
&\qquad + \Big(1 - \frac{\theta_{k,1}}{\gamma_{k,1}} - \frac{\theta_{k,2}}{\gamma_{k,2}}\Big)\nabla^2 f(x^k)(x^{k+1}-x^k) + \frac{1}{2}\Big(1 - \frac{\theta_{k,1}}{\gamma_{k,1}^2} - \frac{\theta_{k,2}}{\gamma_{k,2}^2}\Big)\nabla^3 f(x^k)(x^{k+1}-x^k)^2 \\
&= (1-\theta_{k,1}-\theta_{k,2})(m^k - \nabla f (x^k))\\
&\qquad + \theta_{k,1}( G(z^{k+1,1};\xi^{k+1}) - \nabla f(z^{k+1,1})) + \theta_{k,2}( G(z^{k+1,2};\xi^{k+1}) - \nabla f(z^{k+1,2})) \nonumber\\
&\qquad - (\nabla f(x^{k+1})-\nabla f(x^k) - \nabla^2 f(x^k)(x^{k+1}-x^k)  - \frac{1}{2}\nabla^3 f(x^k)(x^{k+1}-x^k)^2)\nonumber\\
&\qquad + \theta_{k,1}(\nabla f(z^{k+1,1})-\nabla f(x^k) - \frac{1}{\gamma_{k,1}}\nabla^2 f(x^k)(x^{k+1} -x^k) - \frac{1}{2\gamma_{k,1}^2}\nabla^3 f(x^k)(x^{k+1}-x^k)^2)\\
&\qquad + \theta_{k,2}(\nabla f(z^{k+1,2})-\nabla f(x^k) - \frac{1}{\gamma_{k,2}}\nabla^2 f(x^k)(x^{k+1} -x^k) - \frac{1}{2\gamma_{k,2}^2}\nabla^3 f(x^k)(x^{k+1}-x^k)^2)\\
& \overset{\cref{3-z-x}}{=} (1-\theta_{k,1}-\theta_{k,2})(m^k - \nabla f (x^k))\\
&\qquad + \theta_{k,1}( G(z^{k+1,1};\xi^{k+1}) - \nabla f(z^{k+1,1})) + \theta_{k,2}( G(z^{k+1,2};\xi^{k+1}) - \nabla f(z^{k+1,2})) \nonumber\\
&\qquad - (\nabla f(x^{k+1})-\nabla f(x^k) - \nabla^2 f(x^k)(x^{k+1}-x^k)  - \frac{1}{2}\nabla^3 f(x^k)(x^{k+1}-x^k)^2)\nonumber\\
&\qquad + \theta_{k,1}(\nabla f(z^{k+1,1})-\nabla f(x^k) - \nabla^2 f(x^k)(z^{k+1,1} -x^k) - \frac{1}{2}\nabla^3 f(x^k)(z^{k+1,1}-x^k)^2)\\
&\qquad + \theta_{k,2}(\nabla f(z^{k+1,2})-\nabla f(x^k) - \nabla^2 f(x^k)(z^{k+1,2} -x^k) - \frac{1}{2}\nabla^3 f(x^k)(z^{k+1,2}-x^k)^2).
\end{align*}
Taking the squared norm and the expectation with respect to $\xi^{k+1}$ for both sides of this equality, we obtain that for all $a>0$,
\begin{align}
&\mathbb{E}_{\xi^{k+1}}[\|m^{k+1} - \nabla f (x^{k+1})\|^2]\nonumber\\
&= \mathbb{E}_{\xi^{k+1}}[\|\theta_{k,1}( G(z^{k+1,1};\xi^{k+1}) - \nabla f(z^{k+1,1})) + \theta_{k,2}( G(z^{k+1,2};\xi^{k+1}) - \nabla f(z^{k+1,2}))\|^2]\nonumber\\
&\qquad + \|(1-\theta_{k,1}-\theta_{k,2})(m^k - \nabla f (x^k))\nonumber\\
&\qquad - (\nabla f(x^{k+1})-\nabla f(x^k) - \nabla^2 f(x^k)(x^{k+1}-x^k)  - \frac{1}{2}\nabla^3 f(x^k)(x^{k+1}-x^k)^2)\nonumber\\
&\qquad + \theta_{k,1}(\nabla f(z^{k+1,1})-\nabla f(x^k) - \nabla^2 f(x^k)(z^{k+1,1} -x^k) - \frac{1}{2}\nabla^3 f(x^k)(z^{k+1,1}-x^k)^2)\nonumber\\
&\qquad + \theta_{k,2}(\nabla f(z^{k+1,2})-\nabla f(x^k) - \nabla^2 f(x^k)(z^{k+1,2} -x^k) - \frac{1}{2}\nabla^3 f(x^k)(z^{k+1,2}-x^k)^2)\|^2\nonumber\\
&\le 2\theta_{k,1}^2\mathbb{E}_{\xi^{k+1}}[\| G(z^{k+1,1};\xi^{k+1})-\nabla f(z^{k+1,1})\|^2] + 2\theta_{k,2}^2\mathbb{E}_{\xi^{k+1}}[\| G(z^{k+1,2};\xi^{k+1})-\nabla f(z^{k+1,2})\|^2] \nonumber\\
&\qquad  + (1-\theta_{k,1}-\theta_{k,2})^2(1+a)\|m^k - \nabla f (x^k)\|^2\nonumber \\
&\qquad + (1+1/a)\|\nabla f(x^{k+1})-\nabla f(x^k) - \nabla^2 f(x^k)(x^{k+1}-x^k)  - \frac{1}{2}\nabla^3 f(x^k)(x^{k+1}-x^k)^2\nonumber\\
&\qquad - \theta_{k,1}(\nabla f(z^{k+1,1})-\nabla f(x^k) - \nabla^2 f(x^k)(z^{k+1,1} -x^k) - \frac{1}{2}\nabla^3 f(x^k)(z^{k+1,1}-x^k)^2)\nonumber\\
&\qquad - \theta_{k,2}(\nabla f(z^{k+1,2})-\nabla f(x^k) - \nabla^2 f(x^k)(z^{k+1,2} -x^k) - \frac{1}{2}\nabla^3 f(x^k)(z^{k+1,2}-x^k)^2)\|^2\nonumber\\
&\le 2(\theta_{k,1}^2+ \theta_{k,2}^2)\sigma^2  + (1-\theta_{k,1}-\theta_{k,2})^2(1+a)\|m^k - \nabla f (x^k)\|^2 \nonumber\\
&\qquad + 3(1+1/a)\|\nabla f(x^{k+1})-\nabla f(x^k) - \nabla^2 f(x^k)(x^{k+1}-x^k)  - \frac{1}{2}\nabla^3 f(x^k)(x^{k+1}-x^k)^2\|^2\nonumber\\
&\qquad + 3(1+1/a)\theta_{k,1}^2\|\nabla f(z^{k+1,1})-\nabla f(x^k) - \nabla^2 f(x^k)(z^{k+1,1} -x^k) - \frac{1}{2}\nabla^3 f(x^k)(z^{k+1,1}-x^k)^2\|^2\nonumber\\
&\qquad + 3(1+1/a)\theta_{k,2}^2\|\nabla f(z^{k+1,2})-\nabla f(x^k) - \nabla^2 f(x^k)(z^{k+1,2} -x^k) - \frac{1}{2}\nabla^3 f(x^k)(z^{k+1,2}-x^k)^2\|^2,\label{upbd-expect-mk-gradf}
\end{align}
where the equality is due to the first relation in \cref{est-unbias-varbd}, the first inequality is due to $\|u+v\|^2\le(1+a)\|u\|^2+(1+1/a)\|v\|^2$ for all $u,v\in\mathbb{R}^n$ and $a>0$, and the last inequality is due to $\|u+v+w\|^2\le3\|u\|^2+3\|v\|^2+3\|w\|^2$ for all $u,v,w\in\mathbb{R}^n$, and the second relation in \cref{est-unbias-varbd}. In addition, recall from \cref{update-xk} and \cref{3-z-x} that
\begin{align*}
\|x^{k+1} - x^k\|=\eta_k,\quad\|z^{k+1,1}-x^k\|=\eta_k/\gamma_{k,1},\quad\|z^{k+1,2}-x^k\|= \eta_k/\gamma_{k,2}.   
\end{align*}
It then follows from \cref{ineq:3rd-desc} with $(x,y)=(x^k,x^{k+1}),(x^k,z^{k+1,1}),(x^k,z^{k+1,2})$ that
\begin{align*}
&\|\nabla f(x^{k+1})-\nabla f(x^k) - \nabla^2 f(x^k)(x^{k+1}-x^k)  - \frac{1}{2}\nabla^3 f(x^k)(x^{k+1}-x^k)^2\|\le\frac{L_3}{6}\|x^{k+1}-x^k\|^3=\frac{L_3\eta_k^3}{6},\\
&\|\nabla f(z^{k+1,1})-\nabla f(x^k) - \nabla^2 f(x^k)(z^{k+1,1} -x^k) - \frac{1}{2}\nabla^3 f(x^k)(z^{k+1,1}-x^k)^2\|\le\frac{L_3}{6}\|z^{k+1,1}-x^k\|^3=\frac{L_3\eta_k^3}{6\gamma_{k,1}^3},\\
%&\qquad\qquad\qquad\qquad\qquad\qquad\qquad\qquad\qquad\qquad\qquad\qquad\qquad\qquad\qquad\qquad\qquad\qquad\qquad\qquad i=1,2,%\\
&\|\nabla f(z^{k+1,2})-\nabla f(x^k) - \nabla^2 f(x^k)(z^{k+1,2} -x^k) - \frac{1}{2}\nabla^3 f(x^k)(z^{k+1,2}-x^k)^2\|\le\frac{L_3}{6}\|z^{k+1,2}-x^k\|^3=\frac{L_3\eta_k^3}{6\gamma_{k,2}^3}.
\end{align*}
Substituting these inequalities into \cref{upbd-expect-mk-gradf} and letting $a=(\theta_{k,1}+\theta_{k,2})/(1 - \theta_{k,1} - \theta_{k,2})$, we obtain that
\begin{align*}
&\mathbb{E}_{\xi^{k+1}}[\|m^{k+1} - \nabla f (x^{k+1})\|^2]\le (1-\theta_{k,1}-\theta_{k,2})\|m^k - \nabla f (x^k)\|^2\\
&\qquad\qquad\qquad\qquad + \frac{L_3^2\eta_k^6\theta_{k,1}^2}{12\gamma_{k,1}^6(\theta_{k,1}+\theta_{k,2})} + \frac{L_3^2\eta_k^6\theta_{k,2}^2}{12\gamma_{k,2}^6(\theta_{k,1}+\theta_{k,2})} + \frac{L_3^2\eta_k^6}{12(\theta_{k,1}+\theta_{k,2})} + 2(\theta_{k,1}^2+\theta_{k,2}^2) \sigma^2.
\end{align*}    
Hence, the conclusion of this lemma holds as desired.
\end{proof}

We are now ready to prove \cref{thm:3rd-conv}.

\begin{proof}[{Proof of \cref{thm:3rd-conv}}]

It follows from \cref{ineq:1st-desc-1}, \cref{rec:3rd-m}, \cref{def:3rd-Pk}, and \cref{thetak-3rd} that for all $k\ge0$,
\begin{align}
&\E_{\xi^{k+1}}[P_{k+1}]\overset{\cref{def:3rd-Pk}}{=} \E_{\xi^{k+1}}[f(x^{k+1}) + p_{k+1}\|m^{k+1} - \nabla f (x^{k+1})\|^2]\nonumber  \\
&\overset{\cref{ineq:1st-desc-1}\cref{rec:3rd-m}}{\le} f(x^k) - \eta_k \|\nabla f(x^k)\| + 2\eta_k\|\nabla f(x^k) - m^k\| + \frac{L_1}{2}\eta_k^2 + (1-\theta_{k,1}-\theta_{k,2})p_{k+1}\|m^k - \nabla f (x^k)\|^2\nonumber \\
&\qquad + \frac{L_3^2\eta_k^6\theta_{k,1}^2p_{k+1}}{12\gamma_{k,1}^6(\theta_{k,1} + \theta_{k,2})} + \frac{L_3^2\eta_k^6\theta_{k,2}^2p_{k+1}}{12\gamma_{k,2}^6(\theta_{k,1} + \theta_{k,2})} + \frac{L_3^2\eta_k^6p_{k+1}}{12(\theta_{k,1} + \theta_{k,2})} + 2(\theta_{k,1}^2+\theta_{k,2}^2)p_{k+1}\sigma^2\nonumber\\
&\overset{\cref{thetak-3rd}}{\le} f(x^k) - \eta_k \|\nabla f(x^k)\| + 2\eta_k\|\nabla f(x^k) - m^k\| + \frac{L_1}{2}\eta_k^2 + (1-(\theta_{k,1}+\theta_{k,2})/2)p_{k}\|m^k - \nabla f (x^k)\|^2\nonumber \\
&\qquad + \frac{L_3^2\eta_k^6\theta_{k,1}^2p_{k+1}}{12\gamma_{k,1}^6(\theta_{k,1} + \theta_{k,2})} + \frac{L_3^2\eta_k^6\theta_{k,2}^2p_{k+1}}{12\gamma_{k,2}^6(\theta_{k,1} + \theta_{k,2})} + \frac{L_3^2\eta_k^6p_{k+1}}{12(\theta_{k,1} + \theta_{k,2})} + 2(\theta_{k,1}^2+\theta_{k,2}^2)p_{k+1}\sigma^2.\label{3rd-pt-upbd}
\end{align}
By the Young's inequality, one has that
\begin{align*}
2\eta_k\|\nabla f(x^k) - m^k\| \le \frac{(\theta_{k,1} + \theta_{k,2})p_k}{2}\|\nabla f(x^k) - m^k\|^2 + \frac{2\eta_k^2}{(\theta_{k,1} + \theta_{k,2})p_k}\qquad\forall k\ge0,
\end{align*}
which together with \cref{3rd-pt-upbd} implies that for all $k\ge0$,
\begin{align*}
&\E_{\xi^{k+1}}[P_{k+1}]\le f(x^k) + p_k \|m^k - \nabla f (x^k)\|^2 - \eta_k\|\nabla f(x^k)\| + \frac{L_1}{2}\eta_k^2 + \frac{2\eta_k^2}{(\theta_{k,1} + \theta_{k,2})p_k}\\
&\qquad + \frac{L_3^2\eta_k^6\theta_{k,1}^2p_{k+1}}{12\gamma_{k,1}^6(\theta_{k,1} + \theta_{k,2})} + \frac{L_3^2\eta_k^6\theta_{k,2}^2p_{k+1}}{12\gamma_{k,2}^6(\theta_{k,1} + \theta_{k,2})} + \frac{L_3^2\eta_k^6p_{k+1}}{12(\theta_{k,1} + \theta_{k,2})} + 2(\theta_{k,1}^2+\theta_{k,2}^2)p_{k+1}\sigma^2\\
&\overset{\cref{def:3rd-Pk}}{=} P_k - \eta_k\|\nabla f(x^k)\| + \frac{L_1}{2}\eta_k^2 + \frac{2\eta_k^2}{(\theta_{k,1} + \theta_{k,2})p_k}\\
&\qquad + \frac{L_3^2\eta_k^6\theta_{k,1}^2p_{k+1}}{12\gamma_{k,1}^6(\theta_{k,1} + \theta_{k,2})} + \frac{L_3^2\eta_k^6\theta_{k,2}^2p_{k+1}}{12\gamma_{k,2}^6(\theta_{k,1} + \theta_{k,2})} + \frac{L_3^2\eta_k^6p_{k+1}}{12(\theta_{k,1} + \theta_{k,2})} + 2(\theta_{k,1}^2+\theta_{k,2}^2)p_{k+1}\sigma^2.
\end{align*}
Hence, the conclusion of this theorem holds as desired.
\end{proof}

The following lemma provides some useful properties of $\{(\gamma_{k,t},\theta_{k,t})\}_{1\le t\le 2,k\ge0}$ defined in \cref{eta-gamma-3} and \cref{theta-3}, and provides a sequence $\{p_k\}_{k\ge0}$ such that \cref{thetak-3rd} holds.

\begin{lemma}\label{lem:ppt-3}
Let $\{(\gamma_{k,t},\theta_{k,t})\}_{1\le t\le 2, k\ge0}$ be defined in \cref{eta-gamma-3} and \cref{theta-3}, and let $\{p_k\}_{k\ge0}$ be defined as
\begin{align}\label{p-3}
p_k=(k+3)^{1/5} \qquad\forall k\ge0.
\end{align}
Then, \cref{3rd-lss} holds, and it holds that
\begin{align}
&\theta_{k,1} +\theta_{k,2} \in\Big(\frac{1}{(k+3)^{3/5}},\frac{3}{2(k+3)^{3/5}}\Big)\subset(0,1)\qquad\forall k\ge0,\label{ppt-3-sum-theta-square-bd}\\
&\theta_{k,1}^2 \le \frac{4}{(k+3)^{6/5}},\quad \theta_{k,2}^2\le\frac{1}{4(k+3)^{6/5}}\qquad\forall k\ge0.\label{ppt-3-sum-theta-square-bd-2}
\end{align}
In addition, \cref{thetak-3rd} holds.
\end{lemma}

\begin{proof}
Fix any $k\ge0$. By the definition of $\{(\gamma_{k,t},\theta_{k,t})\}_{1\le t\le 2, k\ge0}$ in \cref{eta-gamma-3} and \cref{theta-3}, one can see that \cref{3rd-lss} holds. In addition, it follows from \cref{theta-3} that \cref{ppt-3-sum-theta-square-bd} and \cref{ppt-3-sum-theta-square-bd-2} hold. We next prove \cref{thetak-3rd}. Note from \cref{ppt-3-sum-theta-square-bd} that $\theta_{k,1} +\theta_{k,2}>1/(k+3)^{3/5}$, which implies that
\begin{align}\label{pf-frac-two-theta}
\frac{1-(\theta_{k,1} +\theta_{k,2})/2}{1-(\theta_{k,1} +\theta_{k,2})} = 1 + \frac{1}{2(1/(\theta_{k,1} +\theta_{k,2})-1)} >  1 + \frac{\theta_{k,1} +\theta_{k,2}}{2} >1 + \frac{1}{2(k+3)^{3/5}}.
\end{align}
In addition, using \cref{p-3}, we have that
\begin{align*}
\frac{p_{k+1}}{p_k} \le \Big(1+\frac{1}{k+3}\Big)^{1/5} \le 1 + \frac{1}{5(k+3)},
\end{align*}
where the second inequality is due to $(1+a)^{1/5}\le 1+a/5$ for all $a\ge0$. This together with \cref{pf-frac-two-theta} and the fact that $5(k+3)>2(k+3)^{3/5}$ implies that \cref{thetak-3rd} holds as desired.    
\end{proof}

We are now ready to prove \cref{cor:rate-3}.

\begin{proof}[{Proof of \cref{cor:rate-3}}]
Notice from \cref{lem:ppt-3} that $\{(\eta_k,\gamma_{k,t},\theta_{k,t})\}_{1\le t\le2,k\ge0}$ defined in \cref{eta-gamma-3} and \cref{theta-3} along with $\{p_k\}_{k\ge0}$ defined in \cref{p-3} satisfies the assumptions in \cref{thm:3rd-conv}. Thus, \cref{3rd-stat-upbd} holds. By summing \cref{3rd-stat-upbd} over $k=0,\ldots, K-1$, and taking expectation with respect to  $\{\xi^k\}_{k=0}^{K-1}$, we obtain the following for all $K\ge1$: 
\begin{align}
&\E[P_K] \le \E[P_0] - \sum_{k=0}^{K-1}\eta_k\E[\|\nabla f(x^k)\|] + \sum_{k=0}^{K-1}\Big(\frac{L_1}{2}\eta_k^2 + \frac{2\eta_k^2}{(\theta_{k,1} + \theta_{k,2})p_k}\nonumber\\
&\qquad\qquad\qquad + \frac{L_3^2\eta_k^6\theta_{k,1}^2p_{k+1}}{12\gamma_{k,1}^6(\theta_{k,1} + \theta_{k,2})} + \frac{L_3^2\eta_k^6\theta_{k,2}^2p_{k+1}}{12\gamma_{k,2}^6(\theta_{k,1} + \theta_{k,2})}   + \frac{L_3^2\eta_k^6p_{k+1}}{12(\theta_{k,1} + \theta_{k,2})}+ 2(\theta_{k,1}^2+\theta_{k,2}^2)p_{k+1}\sigma^2\Big).\label{3rd-stat-upbd-0}
\end{align}
Notice from \cref{def:3rd-Pk} that $\E[P_0]=f(x^0)+p_0\E_{\xi^0}[\|G(x^0;\xi^0)-\nabla f(x^0)\|^2]\le f(x^0)+p_0\sigma^2$ and $\E[P_K]\ge f_{\mathrm{low}}$. Also, observe from \cref{p-3} that $p_{k+1}\le 2p_k$ for all $k\ge0$. Using these, \cref{3rd-stat-upbd-0}, and the fact that $\{\eta_k\}_{k\ge0}$ is nonincreasing, we obtain the following for all $K\ge1$:
\begin{align}
&f_{\mathrm{low}} \le f(x^0) + p_0\sigma^2 - \eta_{K-1}\sum_{k=0}^{K-1}\E[\|\nabla f(x^k)\|] + \sum_{k=0}^{K-1}\Big(\frac{L_1}{2}\eta_k^2 + \frac{2\eta_k^2}{(\theta_{k,1} + \theta_{k,2})p_k}\nonumber\\
&\qquad\qquad\qquad + \frac{L_3^2\eta_k^6\theta_{k,1}^2p_k}{6\gamma_{k,1}^6(\theta_{k,1} + \theta_{k,2})} + \frac{L_3^2\eta_k^6\theta_{k,2}^2p_k}{6\gamma_{k,2}^6(\theta_{k,1} + \theta_{k,2})}   + \frac{L_3^2\eta_k^6p_k}{6(\theta_{k,1} + \theta_{k,2})}+ 4(\theta_{k,1}^2+\theta_{k,2}^2)p_k\sigma^2\Big).\label{3rd-stat-upbd-1}
\end{align}
Rearranging the terms of \cref{3rd-stat-upbd-1} and substituting \cref{eta-gamma-3}, \cref{p-3}, \cref{ppt-3-sum-theta-square-bd}, and \cref{ppt-3-sum-theta-square-bd-2}, we obtain the following for all $K\ge5$:
\begin{align}
&\frac{1}{K}\sum_{k=0}^{K-1}\E[\|\nabla f(x^k)\|]\overset{\cref{3rd-stat-upbd-1}}{\le} \frac{f(x^0) - f_{\mathrm{low}} + p_0\sigma^2}{K\eta_{K-1}} + \frac{1}{K\eta_{K-1}}\sum_{k=0}^{K-1}\Big(\frac{L_1}{2}\eta_k^2 + \frac{2\eta_k^2}{(\theta_{k,1} + \theta_{k,2})p_k}\nonumber\\
&\qquad\qquad\qquad + \frac{L_3^2\eta_k^6\theta_{k,1}^2p_k}{6\gamma_{k,1}^6(\theta_{k,1} + \theta_{k,2})} + \frac{L_3^2\eta_k^6\theta_{k,2}^2p_k}{6\gamma_{k,2}^6(\theta_{k,1} + \theta_{k,2})}   + \frac{L_3^2\eta_k^6p_{k+1}}{6(\theta_{k,1} + \theta_{k,2})}+ 4(\theta_{k,1}^2+\theta_{k,2}^2)p_k\sigma^2\Big)\nonumber\\
&\overset{\cref{eta-gamma-3}\cref{p-3}\cref{ppt-3-sum-theta-square-bd}}{=} \frac{(K+2)^{7/10}(f(x^0) - f_{\mathrm{low}} + 3^{1/5}\sigma^2)}{K}\nonumber\\
&\qquad\qquad\qquad + \frac{(K+2)^{7/10}}{K}\sum_{k=0}^{K-1}\Big(\frac{2+10L_3^2/3 +17\sigma^2 }{k+3} + \frac{L_1}{2(k+3)^{7/5}} + \frac{L_3^2}{6(k+3)^{17/5}}\Big)\nonumber\\
&< \frac{2(f(x^0) - f_{\mathrm{low}} + 2\sigma^2)}{K^{3/10}} + \frac{2}{K^{3/10}} \sum_{k=0}^{K-1}\Big(\frac{2+10L_3^2/3 +17\sigma^2 }{k+3} + \frac{L_1}{2(k+3)^{7/5}} + \frac{L_3^2}{6(k+3)^{17/5}}\Big)\nonumber\\
&\le \frac{2(f(x^0) - f_{\mathrm{low}} + 2\sigma^2 + L_1 + L_3^2/120)}{K^{3/10}} + \frac{2(2+ 10L_3^2/3 +17\sigma^2)\ln(2K/5+1)}{K^{3/10}}\nonumber\\
&< 4( f(x^0) - f_{\mathrm{low}} + 19\sigma^2 + L_1  + 4L_3^2 + 2) K^{-3/10}\ln K \overset{\cref{M3}}{=} M_3 K^{-3/10}\ln K,\label{upbd-expec-gradfs}
\end{align}
where the second inequality is due to $(K+2)^{7/10}< 2K^{7/10}$ for all $K\ge5$, the third inequality follows from $\sum_{k=0}^{K-1}1/(k+3)\le \ln(2K/5+1)$, $\sum_{k=0}^{K-1}1/(k+3)^{7/5}\le (5/2)^{3/5}<2$, and $\sum_{k=0}^{K-1}1/(k+3)^{17/5}\le (5/12)(2/5)^{12/5}<1/20$ due to \cref{upbd:series-ka} with $(a,b)=(3,K+2)$ and $\alpha=1,7/5,17/5$, and the last inequality is due to $1<\ln(2K/5+1)<2\ln K$ for all $K\ge5$. Since $\iota_K$ is uniformly drawn from $\{0,\ldots,K-1\}$, it follows that
\begin{align}\label{upbd-expec-kappa-gradfs}
\E[\|\nabla f(x^{\iota_K})\|]=\frac{1}{K}\sum_{k=0}^{K-1}\E[\|\nabla f(x^k)\|]\overset{\cref{upbd-expec-gradfs}}{\le} M_3 K^{-3/10}\ln K\qquad\forall K\ge5.
\end{align}
In view of \cref{lem:rate-complexity} with $(\alpha,u,v)=(3/10,3\epsilon/(20M_3),K)$, one can see that
\begin{align*}
K^{-3/10}\ln K \le \frac{\epsilon}{M_3}\qquad\forall K\ge \Big(\frac{20M_3}{3\epsilon}\ln\Big(\frac{20M_3}{3\epsilon}\Big)\Big)^{10/3},
\end{align*}
which together with \cref{upbd-expec-kappa-gradfs} proves \cref{expect-kappa-3} as desired. Hence, this theorem holds as desired.
\end{proof}

\subsection{Proof of the main results in \cref{subsec:sfom-mem}}\label{subsec:proof-mem}

In this subsection, we prove \cref{lem:p-sol,lem:pth-rec,thm:pth-conv,cor:rate-p}.

\begin{proof}[{Proof of \cref{lem:p-sol}}]
We first prove that the solution to \cref{pth-lss} is unique and can be written as \cref{theta-kt-explicit}. For convenience, we denote the coefficient matrix in \cref{pth-lss} as
\begin{align*}
\Gamma \doteq \begin{bmatrix}
1/\gamma_{k,1} & 1/\gamma_{k,2} & \cdots & 1/\gamma_{k,q}\\ 
1/\gamma_{k,1}^2 & 1/\gamma_{k,2}^2 & \cdots & 1/\gamma_{k,q}^2\\ 
\vdots&\vdots&\ddots & \vdots\\
1/\gamma_{k,1}^q & 1/\gamma_{k,2}^q & \cdots & 1/\gamma_{k,q}^q
\end{bmatrix}\in\R^{q\times q}.
\end{align*} 
In addition, we define a matrix $V\in\R^{q\times q}$ such that its $t$th row $[v_{t1}\ \cdots\ v_{tq}]$ is defined according to the following polynomial
\begin{align}\label{def:ht-alpha}
h_t(\alpha) \doteq \frac{\alpha\prod_{1\le s\le q, s\neq t}(\alpha-1/\gamma_{k,s})}{1/\gamma_{k,t}\prod_{1\le s\le q, s\neq t}(1/\gamma_{k,t}-1/\gamma_{k,s})}=v_{t1}\alpha + v_{t2}\alpha^2 + \cdots + v_{tq}\alpha^q\qquad \forall 1\le t\le q,    
\end{align}
which satisfies $h_t(1/\gamma_{k,t})=1$ and $h_t(1/\gamma_{k,s})=0$ for all $s\neq t$. By the definitions of $h_t$, $\Gamma$, and $V$, one has that 
\begin{align*}
V\Gamma & = \begin{bmatrix}
\sum_{s=1}^q v_{1s}/\gamma_{k,1}^s& \sum_{s=1}^q v_{1s}/\gamma_{k,2}^s&\cdots& \sum_{s=1}^q v_{1s}/\gamma_{k,q}^s\\
\sum_{s=1}^q v_{2s}/\gamma_{k,1}^s& \sum_{s=1}^q v_{2s}/\gamma_{k,2}^s&\cdots& \sum_{s=1}^q v_{2s}/\gamma_{k,q}^s\\
\vdots&\vdots&\ddots&\vdots\\
\sum_{s=1}^q v_{qs}/\gamma_{k,1}^s& \sum_{s=1}^q v_{qs}/\gamma_{k,2}^s&\cdots& \sum_{s=1}^q v_{qs}/\gamma_{k,q}^s
\end{bmatrix}\\
& =     \begin{bmatrix}
h_1(1/\gamma_{k,1})& h_1(1/\gamma_{k,2}) &\cdots& h_1(1/\gamma_{k,q})\\
h_2(1/\gamma_{k,1})& h_2(1/\gamma_{k,2}) &\cdots& h_2(1/\gamma_{k,q})\\
\vdots&\vdots&\ddots&\vdots\\
h_q(1/\gamma_{k,1})& h_q(1/\gamma_{k,2}) &\cdots& h_q(1/\gamma_{k,q})
\end{bmatrix} = I,
\end{align*}
where $I$ is a $q\times q$ identity matrix. Therefore, we have that $V=\Gamma^{-1}$. In view of this and the definition of $V$, one can see that the solution to \cref{pth-lss} is unique and can be written as
\begin{align}\label{pf-exp-thetakt-ht}
\begin{bmatrix}
\theta_{k,1}\\
\theta_{k,2}\\
\vdots\\
\theta_{k,q}
\end{bmatrix} = V\begin{bmatrix}
1\\
1\\
\vdots\\
1
\end{bmatrix}  = \begin{bmatrix}
h_1(1)\\
h_2(1)\\
\vdots\\
h_q(1)
\end{bmatrix},
\end{align}
which together with the definition of $h_t$ implies that \cref{theta-kt-explicit} holds as desired.

We then prove that $\theta_{k,t}>0$ for all odd $t$ and $\theta_{k,t}<0$ for all even $t$. Since $0<\gamma_{k,q}<\cdots<\gamma_{k,1}<1$, it follows that
\begin{align*}
\mathrm{sgn}\Big(\prod_{1\le s\le q,s\neq t}(1-1/\gamma_{k,s})\Big) = (-1)^{q-1},\quad \mathrm{sgn}\Big(\prod_{1\le s\le q,s\neq t}(1/\gamma_{k,t}-1/\gamma_{k,s})\Big) = (-1)^{q-t},
\end{align*}
which along with \cref{theta-kt-explicit} implies that $\mathrm{sgn}(\theta_{k,t})=(-1)^{t-1}$, and thus, $\theta_{k,t}>0$ for all odd $t$ and $\theta_{k,t}<0$ for all even $t$.

We next prove \cref{sum-q-theta}. Recall from \cref{pf-exp-thetakt-ht} that $\theta_{k,t}=h_t(1)$ for each $1\le t\le q$.
%where $\{c_{st}\}_{1\le s\le q}$ are coefficients of the polynomial $h_t(\alpha)$. 
Let $h(\alpha)\doteq\sum_{t=1}^qh_t(\alpha)$. It then follows that 
\begin{align}\label{sum-theta-kt}
\sum_{t=1}^q\theta_{k,t} = \sum_{t=1}^q h_t(1) = h(1).    
\end{align}
Also, in view of the fact that the polynomial $h_t(\alpha)$ satisfies $h_t(1/\gamma_{k,t})=1$ and $h_t(1/\gamma_{k,s})=0$ for all $1\le s\le q$ and $s\neq t$, one can see that $h(\alpha)$ satisfies $h(0)=0$ and $h(1/\gamma_{k,t})=1$ for all $1\le t\le q$. Using \cref{lem:lag-inter} and the fact that $1/\gamma_{k,t}$, $1\le t\le q$, take distinct values, we obtain that $h(\alpha)$ is unique and can be expressed as 
\begin{align*}
h(\alpha) = 1- \frac{\prod_{t=1}^q(1/\gamma_{k,t}-\alpha)}{\prod_{t=1}^q1/\gamma_{k,t}},
\end{align*}
which along with \cref{sum-theta-kt} proves \cref{sum-q-theta} as desired. Hence, this lemma holds as desired.
\end{proof}

\begin{proof}[{Proof of \cref{lem:pth-rec}}]
Fix any $k\ge0$. Notice from \cref{update-zk} with $q=p-1$ that
\begin{align}\label{p-z-x}
z^{k+1,t} - x^k = \frac{1}{\gamma_{k,t}}(x^{k+1} - x^k)\qquad\forall 1\le t\le p-1.
\end{align}
By this, $q=p-1$, \cref{update-mk}, and \cref{pth-lss}, one has that 
\begin{align*}
&m^{k+1} - \nabla f (x^{k+1}) \overset{ \cref{update-mk}}{=} \Big(1 - \sum_{t=1}^{p-1}\theta_{k,t}\Big)m^k + \sum_{t=1}^{p-1} \theta_{k,t}G(z^{k+1,t};\xi^{k+1}) - \nabla f (x^{k+1}) \\
&\overset{\cref{pth-lss}}{=} \Big(1 - \sum_{t=1}^{p-1}\theta_{k,t}\Big)m^k + \sum_{t=1}^{p-1} \theta_{k,t}G(z^{k+1,t};\xi^{k+1}) - \nabla f(x^{k+1})\\
&\qquad + \sum_{r=2}^{p}\frac{1}{(r-1)!}\Big(1-\sum_{t=1}^{p-1}\frac{\theta_{k,t}}{\gamma_{k,t}^{r-1}}\Big)\nabla^r f(x^k)(x^{k+1}-x^k)^{r-1}\\
&= \Big(1 - \sum_{t=1}^{p-1}\theta_{k,t}\Big)(m^k - \nabla f (x^k)) + \sum_{t=1}^{p-1}\theta_{k,t}( G(z^{k+1,t};\xi^{k+1}) - \nabla f(z^{k+1,t}))\\
&\qquad - \Big(\nabla f(x^{k+1}) - \sum_{r=1}^p\frac{1}{(r-1)!}\nabla^r f(x^k)(x^{k+1}-x^k)^{r-1}\Big)\nonumber\\
&\qquad + \sum_{t=1}^{p-1}\theta_{k,t}\Big(\nabla f(z^{k+1,t})- \sum_{r=1}^p\frac{1}{(r-1)!\gamma_{k,t}^{r-1}}\nabla^r f(x^k)(x^{k+1}-x^k)^{r-1}\Big)\\
& \overset{\cref{p-z-x}}{=} \Big(1 - \sum_{t=1}^{p-1}\theta_{k,t}\Big)(m^k - \nabla f (x^k)) + \sum_{t=1}^{p-1}\theta_{k,t}(G(z^{k+1,t};\xi^{k+1}) - \nabla f(z^{k+1,t}))\\
&\qquad - \Big(\nabla f(x^{k+1}) - \sum_{r=1}^p\frac{1}{(r-1)!}\nabla^r f(x^k)(x^{k+1}-x^k)^{r-1}\Big)\nonumber\\
&\qquad + \sum_{t=1}^{p-1}\theta_{k,t}\Big(\nabla f(z^{k+1,t})- \sum_{r=1}^p\frac{1}{(r-1)!}\nabla^r f(x^k)(z^{k+1,t}-x^k)^{r-1}\Big).
\end{align*}
Taking the squared norm and the expectation with respect to $\xi^{k+1}$ for both sides of this equality, we obtain the following for all $a>0$,
\begin{align}
&\mathbb{E}_{\xi^{k+1}}[\|m^{k+1} - \nabla f (x^{k+1})\|^2]\nonumber\\
&=\mathbb{E}_{\xi^{k+1}}\Big[\Big\|\sum_{t=1}^{p-1}\theta_{k,t}( G(z^{k+1,t};\xi^{k+1}) - \nabla f(z^{k+1,t}))\Big\|^2\Big]\nonumber\\
&\qquad + \Big\|\Big(1 - \sum_{t=1}^{p-1}\theta_{k,t}\Big)(m^k - \nabla f (x^k))\nonumber\\
&\qquad - \Big(\nabla f(x^{k+1}) - \sum_{r=1}^p\frac{1}{(r-1)!}\nabla^r f(x^k)(x^{k+1}-x^k)^{r-1}\Big)\nonumber\\
&\qquad + \sum_{t=1}^{p-1}\theta_{k,t}\Big(\nabla f(z^{k+1,t})- \sum_{r=1}^p\frac{1}{(r-1)!}\nabla^r f(x^k)(z^{k+1,t}-x^k)^{r-1}\Big)\Big\|^2\nonumber\\
&\le (p-1)\sum_{t=1}^{p-1}\theta_{k,t}^2\E_{\xi^{k+1}}[\|G(z^{k+1,t};\xi^{k+1}) - \nabla f(z^{k+1,t})\|^2] + \Big(1-\sum_{t=1}^{p-1}\theta_{k,t}\Big)^2(1+a)\|m^k - \nabla f (x^k)\|^2\nonumber\\
&\qquad + (1+1/a)\Big\|\nabla f(x^{k+1}) - \sum_{r=1}^p\frac{1}{(r-1)!}\nabla^r f(x^k)(x^{k+1}-x^k)^{r-1} \nonumber\\
&\qquad -\sum_{t=1}^{p-1}\theta_{k,t}\Big(\nabla f(z^{k+1,t})- \sum_{r=1}^p\frac{1}{(r-1)!}\nabla^r f(x^k)(z^{k+1,t}-x^k)^{r-1}\Big)\Big\|^2\nonumber\\
&\le (p-1)\sigma^2\sum_{t=1}^{p-1}\theta_{k,t}^2 + \Big(1-\sum_{t=1}^{p-1}\theta_{k,t}\Big)^2(1+a)\|m^k - \nabla f (x^k)\|^2 \nonumber\\
&\qquad + p(1+1/a)\Big\|\nabla f(x^{k+1}) - \sum_{r=1}^p\frac{1}{(r-1)!}\nabla^r f(x^k)(x^{k+1}-x^k)^{r-1}\Big\|^2\nonumber\\
&\qquad + p(1+1/a)\sum_{t=1}^{p-1}\theta_{k,t}^2\Big\|\nabla f(z^{k+1,t})- \sum_{r=1}^p\frac{1}{(r-1)!}\nabla^r f(x^k)(z^{k+1,t}-x^k)^{r-1}\Big\|^2,\label{p-upbd-mk1-gfk1}
\end{align}
where the equality is due to the first relation in \cref{est-unbias-varbd}, the first inequality is due to $\|\sum_{t=1}^{p-1}u_t\|^2\le(p-1)\sum_{t=1}^{p-1}\|u_t\|^2$ for all $u_t\in\mathbb{R}^n$, $1\le t\le p-1$ and $\|u+v\|^2\le(1+a)\|u\|^2+(1+1/a)\|v\|^2$ for all $u,v\in\mathbb{R}^n$ and $a>0$, and the last inequality follows from $\|\sum_{t=1}^pu_t\|^2\le p\sum_{t=1}^p\|u_t\|^2$ for all $u_t\in\mathbb{R}^n$, $1\le t\le p$, and the second relation in \cref{est-unbias-varbd}. In addition, recall from \cref{update-xk} and \cref{p-z-x} that
\begin{align*}
\|x^{k+1} - x^k\|=\eta_k,\quad\|z^{k+1,t}-x^k\|=\eta_k/\gamma_{k,t}\qquad\forall 1\le t\le p-1.   
\end{align*}
It then follows from \cref{ineq:pth-desc} with $(x,y)=(x^k,x^{k+1})$  and $(x,y)=(x^k,z^{k+1,t})$ for all $1\le t\le p-1$ that
\begin{align*}
&\Big\|\nabla f(x^{k+1}) - \sum_{r=1}^p\frac{1}{(r-1)!}\nabla^r f(x^k)(x^{k+1}-x^k)^{r-1}\Big\|\le \frac{L_p}{p!} \|x^{k+1}-x^k\|^{p}  = \frac{L_p}{p!}\eta_k^p,\\
&\Big\|\nabla f(z^{k+1,t}) - \sum_{r=1}^p\frac{1}{(r-1)!}\nabla^r f(x^k)(z^{k+1,t}-x^k)^{r-1}\Big\|\le \frac{L_p}{p!} \|z^{k+1,t}-x^k\|^{p} =\frac{L_p\eta_k^p}{p!\gamma_{k,t}^p}\quad \forall 1\le t\le p-1.
\end{align*}
Substituting these inequalities into \cref{p-upbd-mk1-gfk1} and letting $a= \sum_{t=1}^{p-1}\theta_{k,t}/(1 - \sum_{t=1}^{p-1}\theta_{k,t})$, we obtain that
\begin{align*}
&\mathbb{E}_{\xi^{k+1}}[\|m^{k+1} - \nabla f (x^{k+1})\|^2]\\
&\le \Big(1-\sum_{t=1}^{p-1}\theta_{k,t}\Big)\|m^k - \nabla f (x^k)\|^2 + \frac{pL_p^2\eta_k^{2p}}{(p!)^2\sum_{t=1}^{p-1}\theta_{k,t}} +  \frac{pL_p^2\eta_k^{2p}}{(p!)^2\sum_{t=1}^{p-1}\theta_{k,t}} \sum_{t=1}^{p-1} \frac{\theta_{k,t}^2}{\gamma_{k,t}^{2p}} + (p-1)\sigma^2\sum_{t=1}^{p-1}\theta_{k,t}^2.
\end{align*}    
Hence, the conclusion of this lemma holds as desired.
\end{proof}

We are now ready to prove \cref{thm:pth-conv}.

\begin{proof}[{Proof of \cref{thm:pth-conv}}]
It follows from \cref{ineq:1st-desc-1}, \cref{rec:pth-m}, \cref{def:pth-Pk}, and \cref{thetak-pth} that 
\begin{align}
&\E_{\xi^{k+1}}[P_{k+1}]\overset{\cref{def:pth-Pk}}{=} \E_{\xi^{k+1}}[f(x^{k+1}) + p_{k+1}\|m^{k+1} - \nabla f (x^{k+1})\|^2]\nonumber  \\
&\overset{\cref{rec:pth-m}\cref{ineq:1st-desc-1}}{\le} f(x^k) - \eta_k \|\nabla f(x^k)\| + 2\eta_k\|\nabla f(x^k) - m^k\| + \frac{L_1}{2}\eta_k^2 + \Big(1-\sum_{t=1}^{p-1}\theta_{k,t}\Big)p_{k+1}\|m^k - \nabla f (x^k)\|^2\nonumber \\
&\qquad +  \frac{pL_p^2\eta_k^{2p}p_{k+1}}{(p!)^2\sum_{t=1}^{p-1}\theta_{k,t}} \Big(1 +\sum_{t=1}^{p-1} \frac{\theta_{k,t}^2}{\gamma_{k,t}^{2p}}\Big) + (p-1)\sigma^2p_{k+1}\sum_{t=1}^{p-1}\theta_{k,t}^2\nonumber\\
&\overset{\cref{thetak-pth}}{\le} f(x^k) - \eta_k \|\nabla f(x^k)\| + 2\eta_k\|\nabla f(x^k) - m^k\| + \frac{L_1}{2}\eta_k^2 + \Big(1-\sum_{t=1}^{p-1}\theta_{k,t}/(p+1)\Big)p_k\|m^k - \nabla f (x^k)\|^2\nonumber \\
&\qquad +  \frac{pL_p^2\eta_k^{2p}p_{k+1}}{(p!)^2\sum_{t=1}^{p-1}\theta_{k,t}} \Big(1 +\sum_{t=1}^{p-1} \frac{\theta_{k,t}^2}{\gamma_{k,t}^{2p}}\Big) + (p-1)\sigma^2p_{k+1}\sum_{t=1}^{p-1}\theta_{k,t}^2.\label{pth-pt-upbd}
\end{align}
By the Young's inequality, one has that
\begin{align*}
2\eta_k\|\nabla f(x^k) - m^k\| \le \frac{p_k\sum_{t=1}^{p-1}\theta_{k,t}}{p+1}\|\nabla f(x^k) - m^k\|^2 + \frac{(p+1)\eta_k^2}{p_k\sum_{t=1}^{p-1}\theta_{k,t}},
\end{align*}
which together with \cref{pth-pt-upbd} implies that 
\begin{align*}
&\E_{\xi^{k+1}}[P_{k+1}] \le f(x^k) + p_k \|m^k - \nabla f (x^k)\|^2 - \eta_k\|\nabla f(x^k)\| + \frac{L_1}{2}\eta_k^2 + \frac{(p+1)\eta_k^2}{p_k\sum_{t=1}^{p-1}\theta_{k,t}}\\
&\qquad\qquad\qquad\qquad  +  \frac{pL_p^2\eta_k^{2p}p_{k+1}}{(p!)^2\sum_{t=1}^{p-1}\theta_{k,t}} \Big(1 +\sum_{t=1}^{p-1} \frac{\theta_{k,t}^2}{\gamma_{k,t}^{2p}}\Big) + (p-1)\sigma^2p_{k+1}\sum_{t=1}^{p-1}\theta_{k,t}^2\\
&\overset{\cref{def:pth-Pk}}{=} P_k - \eta_k\|\nabla f(x^k)\| + \frac{L_1}{2}\eta_k^2 + \frac{(p+1)\eta_k^2}{p_k\sum_{t=1}^{p-1}\theta_{k,t}} +  \frac{pL_p^2\eta_k^{2p}p_{k+1}}{(p!)^2\sum_{t=1}^{p-1}\theta_{k,t}} \Big(1 +\sum_{t=1}^{p-1} \frac{\theta_{k,t}^2}{\gamma_{k,t}^{2p}}\Big) + (p-1)\sigma^2p_{k+1}\sum_{t=1}^{p-1}\theta_{k,t}^2.
\end{align*}
Hence, the conclusion of this theorem holds as desired.
\end{proof}

The following lemma provides some useful properties of $\{(\gamma_{k,t},\theta_{k,t})\}_{1\le t\le p-1,k\ge0}$ defined in \cref{eta-gamma-p} and \cref{theta-p}, and also provides a sequence $\{p_k\}_{k\ge0}$ such that \cref{thetak-pth} holds.
\begin{lemma}\label{lem:pth-theta-p}
Suppose that \cref{asp:pth-smth} hold. Let $\{(\gamma_{k,t},\theta_{k,t})\}_{1\le t\le p-1,k\ge0}$ be defined in \cref{eta-gamma-p} and \cref{theta-p}, and let $\{p_k\}_{k\ge0}$ be defined as 
\begin{align}\label{p-pk}
p_k = (k+p)^{(p-1)/(3p+1)}\qquad\forall k\ge0,
\end{align}
where $p$ is given in \cref{asp:pth-smth}. Then, \cref{pth-lss} holds, and it holds that 
\begin{align}
&\sum_{t=1}^{p-1}\theta_{k,t}\in\bigg[\frac{1}{2(k+p)^{2p/(3p+1)}},\frac{\ln(2p-1)}{(k+p)^{2p/(3p+1)}}\bigg]\subset(0,1)\qquad \forall k\ge0,\label{sum-theta-kt-p}\\
&\theta_{k,t}^2\le \frac{16((p-1)!)^2}{t^2(k+p)^{4p/(3p+1)}}\qquad\forall 1\le t\le p-1,k\ge0.\label{upbd-sq-theta-kt-p}
\end{align}
In addition, \cref{thetak-pth} holds.
\end{lemma}

\begin{proof}
Fix an arbitrary $k\ge0$. In view of \cref{lem:p-sol} and the definition of $\{(\gamma_{k,t},\theta_{k,t})\}_{1\le t\le p-1,k\ge0}$ in \cref{eta-gamma-p} and \cref{theta-p}, we observe that \cref{pth-lss} holds. We next prove \cref{sum-theta-kt-p}. Since \cref{theta-p} is the solution to \cref{pth-lss} with $q=p-1$ and $\{\gamma_{k,t}\}_{1\le t\le p-1}$ specified in \cref{eta-gamma-p}, it follows from \cref{lem:p-sol} that \cref{sum-q-theta} holds. By substituting \cref{eta-gamma-p} into \cref{sum-q-theta}, one has that
\begin{align}\label{sum-theta-gama-spec}
\sum_{t=1}^{p-1} \theta_{k,t} = 1 - \frac{\prod_{t=1}^{p-1}(t(k+p)^{2p/(3p+1)}-1)}{\prod_{t=1}^{p-1}(t(k+p)^{2p/(3p+1)})} = 1 - \prod_{t=1}^{p-1}\Big(1-\frac{1}{t(k+p)^{2p/(3p+1)}}\Big).
\end{align}
Notice that
\begin{align}
\prod_{t=1}^{p-1}\Big(1-\frac{1}{t(k+p)^{2p/(3p+1)}}\Big) \ge 1 - \frac{1}{(k+p)^{2p/(3p+1)}}\sum_{t=1}^{p-1}\frac{1}{t}\ge 1-\frac{\ln(2p-1)}{(k+p)^{2p/(3p+1)}} \ge 1- \frac{\ln(2p-1)}{p^{2p/(3p+1)}} >0,\label{lwbd-t-k-kp}
\end{align}
where the first inequality is due to the Weierstrass product inequality $\prod_{i=1}^m(1-\alpha_i)\ge1-\sum_{i=1}^m\alpha_i$ for all $\alpha_i\in(0,1)$, $1\le i\le m$, the second inequality follows from \cref{upbd:series-ka} with $(a,b,\alpha)=(1,p-1,1)$, the third inequality is due to $k\ge0$, and the last inequality is because $\ln(2p-1)<p^{2p/(3p+1)}$ for all $p\ge2$. We also notice that
\begin{align}
\prod_{t=1}^{p-1}\Big(1-\frac{1}{t(k+p)^{2p/(3p+1)}}\Big)  
& \le \Big(1-\frac{1}{(p-1)(k+p)^{2p/(3p+1)}}\Big)^{p-1}\le 1 - \frac{1}{(k+p)^{2p/(3p+1)}+1} \nonumber\\
& \le 1 - \frac{1}{2(k+p)^{2p/(3p+1)}} <1,\label{upbd-t-k-kp}
\end{align}
where the first inequality is due to $1\le t\le p-1$, the second inequality is because $(1+a)^r\le 1/(1-ra)$ for all $a\in[-1,0]$ and $r\ge0$, and the third inequality follows from $(k+p)^{2p/(3p+1)}\ge p^{2p/(3p+1)}>1$ for all $k\ge0$ and $p\ge2$. In view of \cref{sum-theta-gama-spec}, \cref{lwbd-t-k-kp}, and \cref{upbd-t-k-kp}, we see that \cref{sum-theta-kt-p} holds as desired.

We now prove \cref{upbd-sq-theta-kt-p}. It follows from \cref{theta-p} that
\begin{align}
&|\theta_{k,t}| = \frac{\prod_{s=1}^{p-1}(s(k+p)^{2p/(3p+1)}-1)}{t(k+p)^{2p/(3p+1)}(t(k+p)^{2p/(3p+1)}-1)|\prod_{1\le s\le p-1,s\neq t}((t-s)(k+p)^{2p/(3p+1)})|}\nonumber\\
&=\prod_{s=1}^{p-1}\Big(1 - \frac{1}{s(k+p)^{2p/(3p+1)}}\Big)\cdot\frac{\prod_{1\le s\le p-1,s\neq t}s}{(t(k+p)^{2p/(3p+1)}-1)|\prod_{1\le s\le p-1,s\neq t}(t - s)|}\qquad \forall 1\le t\le p-1.\label{abs-ts-}
%&\le \frac{2(p-1)!}{t(k+p)^{2p/(3p+1)}} \qquad\forall 1\le t\le p-1,
\end{align}
In addition, since $p\ge2$, $t\ge1$, and $k\ge0$, it follows that
\begin{align*}
&\prod_{s=1}^{p-1}\Big(1 - \frac{1}{s(k+p)^{2p/(3p+1)}}\Big)\le 1,\quad \prod_{1\le s\le p-1,s\neq t}s\le (p-1)!,\quad \Big|\prod_{1\le s\le p-1,s\neq t}(t - s)\Big|\ge1,\\
&t(k+p)^{2p/(3p+1)}-1\ge t(k+p)^{2p/(3p+1)}/4\qquad\forall 1\le t\le p-1.
\end{align*}
where the last inequality is due to $t(k+p)^{2p/(3p+1)}\ge p^{2p/(3p+1)}\ge 4/3$. These along with \cref{abs-ts-} imply that
\begin{align*}
|\theta_{k,t}| \le \frac{4(p-1)!}{t(k+p)^{2p/(3p+1)}} \qquad\forall 1\le t\le p-1.    
\end{align*}
Hence, \cref{upbd-sq-theta-kt-p} holds as desired.

We finally prove \cref{thetak-pth}. Using \cref{sum-theta-kt-p}, we obtain that
\begin{align}
\frac{1-\sum_{t=1}^{p-1}\theta_{k,t}/(p+1)}{1-\sum_{t=1}^{p-1}\theta_{k,t}}& =1+\frac{p}{(p+1)(1/\sum_{t=1}^{p-1}\theta_{k,t}-1)}>1+\frac{p\sum_{t=1}^{p-1}\theta_{k,t}}{p+1}\nonumber\\
& \overset{\cref{sum-theta-kt-p}}{>} 1 + \frac{p}{2(p+1)(k+p)^{2p/(3p+1)}}.  \label{lwbd-frac-sum-theta}
\end{align}
In addition, observe from \cref{p-pk} that 
\begin{align}\label{lwbd-frac-pk}
\frac{p_{k+1}}{p_k} \le \Big( 1 + \frac{1}{k + p} \Big)^{(p-1)/(3p+1)} \le 1 + \frac{p-1}{(3p+1)(k+ p)},   
\end{align}
where the second inequality is due to $(1+a)^r\le 1+ra$ for all $a\ge0$ and $r\in[0,1]$. Combining \cref{lwbd-frac-sum-theta} and \cref{lwbd-frac-pk} with the fact that $(3p+1)/(p-1)\ge 2(p+1)/p$ and $k+p\ge (k+p)^{2p/(3p+1)}$, we obtain that \cref{thetak-pth} holds as desired.

\end{proof}

We are now ready to prove \cref{cor:rate-p}.

\begin{proof}[{Proof of \cref{cor:rate-p}}]
Recall from \cref{lem:pth-theta-p} that $\{(\eta_k,\gamma_{k,t},\theta_{k,t})\}_{1\le t\le p-1,k\ge0}$ defined in \cref{eta-gamma-p} and \cref{theta-p} along with $\{p_k\}_{k\ge0}$ defined in \cref{p-pk} satisfies the assumptions in \cref{thm:pth-conv}. Thus, \cref{pth-stat-upbd} holds. By summing \cref{pth-stat-upbd} over $k=0,\ldots, K-1$, and taking expectation with respect to $\{\xi^k\}_{k=0}^{K-1}$, we obtain that for all $K\ge1$,
\begin{align}
&\E[P_K]\le \E[P_0] - \sum_{k=0}^{K-1}\eta_k\E[\|\nabla f(x^k)\|]\nonumber\\
&\qquad\qquad + \sum_{k=0}^{K-1}\Big(\frac{L_1}{2}\eta_k^2 + \frac{(p+1)\eta_k^2}{p_k\sum_{t=1}^{p-1}\theta_{k,t}} +  \frac{pL_p^2\eta_k^{2p}p_{k+1}}{(p!)^2\sum_{t=1}^{p-1}\theta_{k,t}} \Big(1 +\sum_{t=1}^{p-1} \frac{\theta_{k,t}^2}{\gamma_{k,t}^{2p}}\Big) + (p-1)\sigma^2p_{k+1}\sum_{t=1}^{p-1}\theta_{k,t}^2\Big).\label{pth-stat-upbd-0}
\end{align}
Observe from \cref{def:pth-Pk} that $\E[P_0]=f(x^0)+p_0\E_{\xi^0}[\|G(x^0;\xi^0)-\nabla f(x^0)\|^2]\le f(x^0)+p_0\sigma^2$ and $\E[P_K]\ge f_{\mathrm{low}}$. In addition, note from \cref{p-pk} and $p\ge2$ that $p_{k+1}\le 2p_k$ for all $k\ge0$. Using these, \cref{pth-stat-upbd-0}, and the fact that $\{\eta_k\}_{k\ge0}$ is nonincreasing, we obtain that for all $K\ge1$,
\begin{align}
&f_{\mathrm{low}} \le f(x^0) + p_0\sigma^2 - \eta_{K-1}\sum_{k=0}^{K-1}\E[\|\nabla f(x^k)\|]\nonumber\\
&\qquad\qquad + \sum_{k=0}^{K-1}\Big(\frac{L_1}{2}\eta_k^2 + \frac{(p+1)\eta_k^2}{p_k\sum_{t=1}^{p-1}\theta_{k,t}} +  \frac{2pL_p^2\eta_k^{2p}p_k}{(p!)^2\sum_{t=1}^{p-1}\theta_{k,t}} \Big(1 +\sum_{t=1}^{p-1} \frac{\theta_{k,t}^2}{\gamma_{k,t}^{2p}}\Big) + 2(p-1)\sigma^2p_k\sum_{t=1}^{p-1}\theta_{k,t}^2\Big).\label{pth-stat-upbd-1}
\end{align}
Rearranging the terms of this inequality, and substituting \cref{eta-gamma-p}, \cref{p-pk}, \cref{sum-theta-kt-p}, and \cref{upbd-sq-theta-kt-p}, we obtain the following for all $K\ge1$:
\begin{align}
&\frac{1}{K}\sum_{k=1}^{K-1}\E[\|\nabla f(x^k)\|]\overset{\cref{pth-stat-upbd-1}}{\le} \frac{f(x^0) - f_{\mathrm{low}} + p_0\sigma^2}{K\eta_{K-1}} \nonumber\\
&\qquad + \frac{1}{K\eta_{K-1}}\sum_{k=0}^{K-1}\Big(\frac{L_1}{2}\eta_k^2 + \frac{(p+1)\eta_k^2}{p_k\sum_{t=1}^{p-1}\theta_{k,t}} +  \frac{2pL_p^2\eta_k^{2p}p_k}{(p!)^2\sum_{t=1}^{p-1}\theta_{k,t}} \Big(1 +\sum_{t=1}^{p-1} \frac{\theta_{k,t}^2}{\gamma_{k,t}^{2p}}\Big) + 2(p-1)\sigma^2p_k\sum_{t=1}^{p-1}\theta_{k,t}^2\Big)\nonumber\\
&\overset{\cref{eta-gamma-p}\cref{p-pk}\cref{sum-theta-kt-p}\cref{upbd-sq-theta-kt-p}}{=}\frac{(K+p-1)^{(2p+1)/(3p+1)}(f(x^0) - f_{\mathrm{low}} + p^{(p-1)/(3p+1)}\sigma^2)}{K}\nonumber\\
&\qquad + \frac{(K+p-1)^{(2p+1)/(3p+1)}}{K}\sum_{k=0}^{K-1}\Big( \frac{L_1}{2(k+p)^{(4p+2)/(3p+1)}}\nonumber\\
&\qquad + \frac{4pL_p^2}{(p!)^2(k+ p)^{(4p^2-p+1)/(3p+1)}} + \frac{2(p+1 + 32p^{2p}L_p^2 + 16(p-1)^2((p-1)!)^2\sigma^2)}{k+ p}\Big)\nonumber\\
&\le \frac{(K+p-1)^{(2p+1)/(3p+1)}}{K}\Big(f(x^0) - f_{\mathrm{low}} + p\sigma^2 + \frac{(3p+1)L_1}{2(p+1)(p-1/2)^{(p+1)/(3p+1)}}\nonumber\\
&\qquad + \frac{(3p+1)L_p^2}{(p!)^2(p-1)(p-1/2)^{(4p^2-4p)/(3p+1)}}  + 2(p+1 + 32p^{2p}L_p^2 + 16(p-1)^2((p-1)!)^2\sigma^2)\ln\Big(\frac{2K}{2p-1}+1\Big)\Big)\nonumber\\
&\le \frac{(K+p-1)^{(2p+1)/(3p+1)}}{K}\Big(f(x^0) - f_{\mathrm{low}} + p\sigma^2 + \frac{3L_1}{2(p-1/2)^{(p+1)/(3p+1)}}+ \frac{7L_p^2}{(p!)^2(p-1/2)^{(4p^2-4p)/(3p+1)}}\nonumber\\
&\qquad + 2(p+1 + 32p^{2p}L_p^2 + 16(p!)^2\sigma^2)\ln\Big(\frac{2K}{2p-1}+1\Big)\Big),\label{expect-gradf-xr-p}
\end{align}
where the third inequality follows from $\sum_{k=0}^{K-1}1/(k+p)\le \ln(2K/(2p-1)+1)$, $\sum_{k=0}^{K-1}1/(k+p)^{(4p+2)/(3p+1)}\le(3p+1)/(p+1) (p-1/2)^{-(p+1)/(3p+1)}$, and $\sum_{k=0}^{K-1}1/(k+p)^{(4p^2-p+1)/(3p+1)}\le(3p+1)/(4p^2-4p)(p-1/2)^{(4p-4p^2)/(3p+1)}$ due to \cref{upbd:series-ka} with $(a,b)=(p,K+p-1)$ and $\alpha=1,(4p+2)/(3p+1),(4p^2-p+1)/(3p+1)$, and the last inequality is due to $p\ge2$. In addition, notice that for all $K\ge 2p$,
\begin{align*}
&(K+p-1)^{(2p+1)/(3p+1)}\le (2K)^{(2p+1)/(3p+1)} \le 2 K^{(2p+1)/(3p+1)},\\ 
&\Big(p-\frac{1}{2}\Big)^{(p+1)/(3p+1)} \ge 1,\quad  \Big(p-\frac{1}{2}\Big)^{(4p^2-4p)/(3p+1)} \ge1,\\
&1< \ln\Big(\frac{2}{2p-1}+3\Big)\le\ln\Big(\frac{2K}{2p-1}+1\Big)\le\ln(2K+1) \le 2\ln K.  
\end{align*}
Substituting these inequalities into \cref{expect-gradf-xr-p}, we obtain that for all $K\ge 2p$,
\begin{align}
&\frac{1}{K}\sum_{k=1}^{K-1}\E[\|\nabla f(x^k)\|]\nonumber\\
&\le 4\Big(f(x^0) - f_{\mathrm{low}} + p\sigma^2 +\frac{3L_1}{2} + \frac{7L_p^2}{(p!)^2} + 2(p+1 + 32p^{2p}L_p^2 + 16(p!)^2\sigma^2) \Big)K^{-p/(3p+1)}\ln K\nonumber\\
&\overset{\cref{M-p}}{=}M_pK^{-p/(3p+1)}\ln K.\label{upbd-expec-gradfs-p}
\end{align}
Since $\iota_K$ is uniformly drawn from $\{0,\ldots,K-1\}$, it follows that
\begin{align}\label{upbd-expec-kappa-gradfs-p}
\E[\|\nabla f(x^{\iota_K})\|]=\frac{1}{K}\sum_{k=0}^{K-1}\E[\|\nabla f(x^k)\|]\overset{\cref{upbd-expec-gradfs-p}}{\le} M_p K^{-p/(3p+1)}\ln K\qquad\forall K\ge2p.
\end{align}
In view of \cref{lem:rate-complexity} with $(\alpha,u,v)=(p/(3p+1),p\epsilon/((6p+2)M_p),K)$, one can see that
\begin{align*}
K^{-p/(3p+1)}\ln K \le \frac{\epsilon}{M_p}\qquad\forall K\ge \Big(\frac{(6p+2)M_p}{p\epsilon}\ln\Big(\frac{(6p+2)M_p}{p\epsilon}\Big)\Big)^{(3p+1)/p},
\end{align*}
which together with \cref{upbd-expec-kappa-gradfs-p} proves \cref{expect-kappa-p}. Hence, the conclusion of this theorem holds as desired.
\end{proof}

\section{Concluding remarks}\label{sec:cr}

In this paper, we establish upper bounds for the sample and first-order operation complexity of SFOMs for finding an $\epsilon$-stochastic stationary point of problem \cref{pb:uc} assuming the Lipschitz continuity of the $p$th-order derivative for any $p\ge2$. It should be mentioned that the development of upper complexity bounds in this paper is independent of those in \cite{cutkosky2019momentum, fang2018spider} under the mean-squared smoothness condition stated in \cref{asp:ave-smt}. Indeed, the mean-squared smoothness is an assumption imposed on the stochastic gradient estimator $G(\cdot;\xi)$ rather than the function $f$ itself. In \cref{apx:ctexp}, we provide an example with discussion showing that an unbiased stochastic gradient estimator $G(\cdot;\xi)$ with bounded variance can violate \cref{asp:ave-smt} when the gradient and arbitrarily higher-order derivatives of $f$ are Lipschitz continuous.

In addition, our proposed acceleration technique, which exploits higher-order smoothness, are orthogonal to the recent developments in the complexity results of SFOMs for composite optimization and constrained optimization discussed in \cref{sec:intro}. Thus, it would be interesting to extend our techniques to develop accelerated SFOMs with complexity guarantees for more general stochastic optimization problems under higher-order smoothness, which we leave as a direction for future research.

\appendix

\section{An example concerning mean-squared smoothness}\label{apx:ctexp}

The following proposition provides a necessary and sufficient condition for the mean-squared smoothness condition in \cref{asp:ave-smt}, when the stochastic gradient estimator $G(\cdot;\xi)$ takes a special form.

\begin{proposition}\label{pro:special}
Suppose that Assumptions \ref{asp:basic}(a), (b), and \ref{asp:pth-smth} hold. Let
\begin{align}\label{exp-Gxxi}
G(x;\xi) \doteq \nabla f(x) + \xi g(x)\qquad\forall x\in\mathbb{R}^n,   
\end{align}
where $\xi$ follows the standard normal distribution with mean $0$ and variance $1$, and $g:\mathbb{R}^n\to\mathbb{R}^n$ is a mapping satisfying $\|g(x)\|\le \sigma$ for all $x\in\mathbb{R}^n$ with some $\sigma>0$. Then, $G(\cdot;\xi)$ satisfies \cref{asp:basic}(c), and moreover, \cref{asp:ave-smt} holds if and only if $g$ is Lipschitz continuous in $\R^n$.
\end{proposition}

\begin{proof}
Since $\xi$ follows the standard normal distribution, we have $\E[\xi]=0$ and $\E[\xi^2]=1$. It then follows from \cref{exp-Gxxi} and $\|g(x)\|\le \sigma$ for all $x\in\mathbb{R}^n$ that \cref{asp:basic}(c) holds. In addition, observe that
\begin{align*}
\mathbb{E}_\xi[\|G(y;\xi)-G(x;\xi)\|^2] & \overset{\cref{exp-Gxxi}}{=} \mathbb{E}_\xi[\|\nabla f(y)-\nabla f(x) + \xi(g(y) - 
g(x))\|^2]\nonumber\\
& = \|\nabla f(y) - \nabla f(x)\|^2 + \|g(y) - g(x)\|^2,
\end{align*}
where the second equality follows from $\E[\xi]=0$. This immediately implies that \cref{asp:ave-smt} holds if and only if $g$ is Lipschitz continuous in $\R^n$. Hence, the conclusion of this proposition holds as desired.
\end{proof}

As a consequence of \cref{pro:special}, we conclude that when $G(\cdot;\xi)$ satisfies \cref{exp-Gxxi} with $g$ being bounded but not Lipschitz continuous in $\R^n$ (e.g., $g(x)=\min\{\sqrt{\|x\|},1\}$) or even discontinuous at some points in $\R^n$, \cref{asp:basic}(c) holds, but the mean-squared smoothness condition in \cref{asp:ave-smt} is violated.

%\section{Future directions}
%This work suggests many directions for future research. First, developing methods to exploit high-order smoothness for solving more general optimization problems, such as composite optimization and constrained optimization, would be interesting. Second, we hope our theoretical study can inspire new momentum schemes that can be integrated into optimizers for training deep neural networks. Finally, upper and lower bounds for the sample complexity of achieving approximate stationarity when minimizing an infinitely differentiable function remain open.

\bibliographystyle{abbrv}
\bibliography{ref}

\end{document}